\documentclass[12pt]{amsart}

\usepackage{amssymb,latexsym,graphics,enumerate}
\usepackage[mathscr]{eucal}
\usepackage{amsmath,amsfonts,amsthm,amssymb}
\usepackage{color}
\usepackage[left=1in,right=1in,top=1in,bottom=1in]{geometry}
\usepackage{graphicx}
\usepackage[mathscr]{eucal}
\usepackage[dvipsnames]{xcolor}
\usepackage[colorlinks=true,pdfstartview=FitV,
linkcolor=red,citecolor=ForestGreen, urlcolor=blue]{hyperref}

\newcommand{\bx}{{\mathbf x}}
\newcommand{\by}{{\mathbf y}}
\newcommand{\bz}{{\mathbf z}}
\newcommand{\bu}{{\mathbf u}}
\newcommand{\bv}{{\mathbf v}}
\newcommand{\bw}{{\mathbf w}}
\newcommand{\bm}{{\mathbf m}}

\def\scrM{\mathscr{M}}
\def\scrN{\mathscr{N}}
\def\scrW{\mathscr{W}}

\renewcommand{\geq}{\geqslant}
\renewcommand{\leq}{\leqslant}
\newcommand{\D}{\mbox{\textnormal{d}}}

\newcommand\dbx{{\D}\bx}

\newcommand\dby{{\D}\by}

\newcommand\dbv{{\D}\bv}

\newcommand\dbw{{\D}\bw}
\newcommand\dbz{{\D}\bz}

\newtheorem{thm}{Theorem}

\newtheorem{lem}{Lemma}

\newtheorem{rmk}{Remark}
\newtheorem{exm}{Example}

\numberwithin{equation}{section}
\numberwithin{thm}{section}
\numberwithin{defn}{section}
\numberwithin{lem}{section}
\numberwithin{pro}{section}
\numberwithin{rmk}{section}
\numberwithin{exm}{section}
\numberwithin{Q}{section}

\newcommand{\rr}{\mathbb{R}}
\newcommand{\rt}{\mathbb{T}}

\newcommand{\lan}{\langle}
\newcommand{\ran}{\rangle}
\newcommand{\be}{\begin{eqnarray*}}
\newcommand{\bel}{\begin{eqnarray}}
\newcommand{\ee}{\end{eqnarray*}}
\newcommand{\eel}{\end{eqnarray}}
\newcommand{\ba}{\begin{aligned}}
\newcommand{\ea}{\end{aligned}}

\newcommand{\al}{\alpha}
\newcommand{\bet}{\beta}
\newcommand{\gam}{\gamma}
\newcommand{\na}{\nabla}

\newcommand{\pa}{\partial}

\def\d{{\mbox{\text{d}}}}
\def\e{{\sf{e}}}

\DeclareMathOperator{\trace}{Tr}
\DeclareMathOperator{\diver}{div}

\def\hf{\frac{1}{2}}

\def\dt{\text{d}t}
\def\ddt{\frac{\text{d}}{\dt}}

\def\etaS{\eta} % spectral gap of symmetric part

\def\cm{c_-}

\def\u{{\mathbf u}}

\def\bs{{\mathbf s}}

\def\DelMP{\Delta_{\scrM} \Phi} 
\def\DelNP{\Delta_{\scrN} \Phi} 
\def\DelMA{\Delta_{\scrM} \APhi}
\def\DelW{\Delta_{\scrW}} %{\Delta_\wrho}
\def\DelPh{\Delta \Phi}
\def\delD{\delta D} % diameter
\def\delV{\delta V} % diameter of velocities

\def\delbu{\delta {\mathbf u}}
\def\tc{}
\def\St{D(t)}  % indicate the dependence of support (t)
\def\wrho{w}
\def\APhi{A}
\def\aphi{a}

\def\ux{{\sf x}}
\def\bux{{{\bf x}}}

\def\Min{M}
\newcommand{\dzA}[1]{{\mathrm deg}_{#1}(\APhi)}

\newcommand{\dzPD}[1]{{\mathrm deg}_{#1}(\Phi(D(t)))}
\newcommand\delE{\delta E} % energy fluctutaions incl. self-intercations
\newcommand{\game}[1]{\mbox{}}
\newcommand\delEz{\delE_0}
\newcommand\delVz{\delV_0}
\newcommand\zetaM{\zeta_{{}_\scrM}}
\newcommand\buinf{\overline{\bu}_\infty}
\newcommand\bvinf{\overline{\bv}_\infty}
\newcommand\rate{\nu} %instead of \eta
\newcommand{\test}{\varphi} % test function

\begin{document}
%%%%%%%%%

\title[A game of alignment: collective behavior of multi-species]{A game of alignment:\\collective behavior of multi-species}

\author{Siming He}
\address{Siming He\newline Department of Mathematics\newline Duke University, Durham NC}
\email{simhe@duke.edu}

\author{Eitan Tadmor}
\address{Eitan Tadmor\newline Department of Mathematics, Institute for Physical Sciences \& Technology \newline and \newline Center for Scientific Computation and Mathematical Modeling (CSCAMM)\newline University of Maryland, College Park, MD}
\email{tadmor@umd.edu}

\date{\today}



%\thanks{\textbf{Acknowledgment.} Research was supported in part by NSF grants 
%	DMS16-13911, RNMS11-07444 (KI-Net) and ONR grant N00014-1812465 (ET)
%}

\vspace*{-1.1cm}

\begin{abstract} 
We study the (hydro-)dynamics of multi-species driven by alignment.  What distinguishes the different  species  is the protocol of their interaction with the rest of the crowd: the collective motion is described by different \emph{communication kernels}, $\phi_{\alpha\beta}$,  between the crowds in species $\alpha$ and $\beta$. We show that flocking of the overall crowd emerges provided the communication array between species forms a \emph{connected graph}. 
In particular,  the crowd within each species need \emph{not}  interact with its own kind, i.e., $\phi_{\al\al}=0$; different species which are  engaged  in such `game' of alignment require a connecting path for propagation of information which will lead to the flocking of overall crowd.
The same methodology applies to multi-species aggregation dynamics governed by first-order alignment: connectivity implies concentration around an emerging consensus.
\end{abstract}

\maketitle
\setcounter{tocdepth}{1}%{1}
\vspace*{-1.1cm}
\tableofcontents

\section{Multi-species dynamics --- statement of main results}

%%%%%%%%%%%%%%%%%%%%%%%%%
\subsection{The (hydro-)dynamics of multi-species}
%%%%%%%%%%%%%%%%%%%%%%%%%
We study the (hydro-)dynamics of multi-species driven by environmental averaging. The `environment' consists of agents, each is identified by  a position/velocity pair
$(\bx^i_\al,\bv^i_\al)\in (\rr^d, \rr^d)$. The indexing $\{\cdot\}_\al^i$ signifies  agent ``$i$'' in a species ``$\alpha$''. What distinguishes one species from another is the way they interact with the environment: let $\phi_{\alpha\beta}\geq 0$ be the \emph{communication kernel} between species $\alpha$ and $\beta$, then the dynamics describes the collective motion of agents, each of which aligns its velocity to a \emph{weighted average} of  velocities of neighboring agents --- both from its own as well as other species,
\[
\left\{
\begin{split}
& \ \ \dot{\bx}^i_\al=\bv_\al^i,\\
& \ \ \dot{\bv}_\al^i=\sum_{\beta\in \mathcal{I}}\frac{1}{N_\beta}\sum_{j=1}^{N_\beta} \phi_{\alpha\beta}(\tc|\bx^j_\beta-\bx^i_\al|)(\bv^j_\beta-\bv^i_\al),
\end{split}
\right.
 \qquad i\in 1,2,..., N_\al, \quad \al\in\mathcal I,
\]
 subject to initial data $(\bx^i_\al, \bv^i_\al)\big|_{t=0}=(\bx_{{\al}0}^i,\bv_{{\al}0}^i)$. Here $N_\alpha$ is the size of the species $\al\in {\mathcal I}$, where ${\mathcal I}$ is a (possibly infinite) index set for the different species. The large-crowd dynamics, 
 $N_{\alpha\in{\mathcal I}}\gg 1$,  is captured by the hydrodynamic description\footnote{Unless otherwise stated, all integrals are taken over ${\mathbb R}^d$.}, consult section \ref{sec:hydrodynamics},
\begin{equation}\label{Hydrodynamic_Flocking_eqs}
\left\{
\begin{split}
& \ \  \pa_t \rho_\al +\na\cdot(\bu_\al\rho_\al)=0; \\
& \ \  \pa_t (\rho_\al \bu_\al)+\na\cdot (\rho_\al \bu_\al\otimes \bu_\al)=\sum_{\beta\in \mathcal{I}} \int \hspace*{-0.1cm}\phi_{\alpha\beta}(\tc|\bx-\by|)\big(\bu_\beta(\by)-\bu_\alpha(\bx)\big)\rho_\alpha(\bx)\rho_\beta(\by)\dby. 
\end{split}\right. 
\end{equation}
Each of the different species is identified by a pair of density/velocity $(\rho_\al, \bu_\al)$, subject to initial condition $(\rho_\al,\bu_\al)\big|_{t=0}=(\rho_{{\al}0},\bu_{{\al}0})\in L_+^1(\rr^d)\times W^{1,\infty}(\rr^d),\enskip \forall\al \in \mathcal{I}$. 
There are two extreme cases: when $\phi_{\alpha\beta}\equiv \phi$ the crowd consists of a single species driven by the same communication kernel
\[
\left\{
\begin{split}
& \ \ \pa_t\rho+\na\cdot(\rho \bu) =0, \\    
& \ \  \pa_t(\rho\bu)+ \nabla\cdot(\rho\bu\otimes \bu)  =\int \phi(|\bx-\by|)(\bu(t,\by)-\bu(t,\bx))\rho(t,\bx)\rho(t,\by)\dby.
\end{split}
\right.
\]
For the large literature on the single species hydrodynamics (as well as discrete dynamics), we refer to \cite{bellomo2017-19active} and the references therein.
 When $\phi_{\alpha\beta}=\phi\delta_{\alpha\beta}$, the crowd of \eqref{Hydrodynamic_Flocking_eqs} splits into independent species driven by the same communication kernel, thus we end up with identical copies of the single species dynamics.  
In this paper we study all the intermediate cases which involve  a genuine \emph{multi}-species dynamics, driven by \emph{symmetric} communication array of radial decreasing kernels, 
$\Phi=\{\phi_{\al\bet}\}$,
\begin{equation}\label{eq:symm}
\phi_{\al\bet}=\phi_{\bet\al}\geq 0, \qquad \phi_{\al\bet} \ \text{are radial and decreasing}.
\end{equation}

%%%%%%%%%%%%%%%%%%%%%%%%%
\subsection{Smooth solutions must flock}
%%%%%%%%%%%%%%%%%%%%%%%%%
Recall that the long time behavior for the single-species model  
is dictated by the  communication kernel  $\phi$, \cite{TanTadmor, HeTadmor17}: if the communication kernel $\phi$ admits a Pareto-type `fat-tail'' decay\footnote{And in a slightly more general setup --- if $\phi$ is global  in the sense that $\displaystyle \int^\infty \hspace*{-0.2cm}\phi(r)\D r = \infty$.}, $\phi(r) \gtrsim (1+r)^{-\theta}$ with $\theta \leq 1$ ,  then ``smooth solutions must flock'', namely, strong solutions of the single-species model exhibit \emph{flocking behavior} as $\displaystyle \max_{\bx\in \mathrm{supp}\,\{\rho(t,\cdot)\}}|\bu(t,\bx)-\buinf|  \stackrel{t\rightarrow \infty}{\longrightarrow}0$.

This brings us to our first main result regarding the large-time behavior of the multi-species dynamics. Let $\Phi(r):=\{\phi_{\al\bet}(r)\}_{\al,\bet\in{\mathcal I}}$ denote the array of communication kernels associated with  \eqref{Hydrodynamic_Flocking_eqs}. The main feature here is that flocking of multi-species dynamics does \emph{not} require direct, global communication among all species --- we allow $\phi_{\al\bet}(r)$ to vanish, indicating lack of communication between some  species $\al$ and $\bet$.
Instead, what matters is a minimal requirement  that the communication among  species forms a  \emph{connected network} in the sense that there is a connecting path which propagates the  information of alignment between every pair of species. To this end, we introduce the \emph{weighted} graph Laplacian associated with $\Phi(r)$,
\begin{equation}\label{eq:PhiLap}
(\DelMP(r))_{\al\bet}:= \left\{\begin{array}{ll} -\phi_{\al\bet}(r)\sqrt{M_{\al}M_\bet}, & \al\neq \bet;\\ \\
{\displaystyle \mathop{\sum}_{\gamma\neq \al}}\phi_{\al\gamma}(r)M_\gamma, & \al=\bet,\end{array}\right. 
\end{equation}
where  the weights,  $\scrM:=\{M_\al\}_{\al\in{\mathcal I}}$, consist  of the masses of the different species which are constant in time,
\[
M_\al:=\int \rho_{\al 0}(\bx)\dbx \equiv \int \rho_\al(t,\bx)\dbx>0.
\]
Its properties  are outlined in section \ref{sec:wFiedler} below. In particular, the communication array $\Phi(r)$ forms a connected graph as long as  its second eigenvalue $\lambda_2\big(\DelMP(r)\big)>0$.  Our main result shows that inter-species connectivity implies the flocking behavior of the whole crowd.
 
\begin{thm}[{\bf Strong solutions must flock}]\label{THM_1}\mbox{}\newline
Let $(\rho_{{\al}}(t,\cdot),\bu_{{\al}}(t,\cdot))\in (L^\infty\cap L_+^{1}(\rr^d))\times W^{1,\infty}(\rr^d), \ \al \in \mathcal{I}$ be a strong solution of the multi-species dynamics  \eqref{Hydrodynamic_Flocking_eqs}, subject to compactly supported  initial conditions $(\rho_{\al 0},\bu_{\al 0})$ with finite velocity fluctuations
\[
\delVz:=\max_{\al,\bet\in{\mathcal I}}\sup_{\bx,\by\in{S_0}}|\bu_{\al 0}(\bx)-\bu_{\bet 0}(\by)| <\infty, \qquad {\mathcal S}_0:= \cup_\al \text{supp}\{\rho_{\al 0}(\cdot)\}.
\]
Assume that the communication array $\Phi(r)=\{\phi_{\al\bet}(r)\}_{\al,\bet\in{\mathcal I}}$ satisfies a Pareto-type `fat-tail' connectivity condition
\begin{align}\label{eq:Pareto}
\lambda_2(\DelMP(r)) \gtrsim \frac{1}{(1+r)^\theta}, \ \ \theta<1. 
\end{align}
Then the support, $\displaystyle {\mathcal S}(t):= \cup_\al \textrm{supp}\{\rho_\al(t,\cdot)\}$, remains within a finite diameter $D_\infty< \infty$ (depending on $1-\theta, M, \delVz$), and the different species  flock towards a limiting velocity  $\buinf$, 
\begin{equation}\label{eq:revisit}
\sum_{\al\in\mathcal{I}}\int |\bu_\alpha(t,\bx)-\buinf|^2\rho_\alpha(t,\bx)\d\bx \leq \sum_{\al\in\mathcal{I}}\int |\bu_{\alpha 0}(\bx)-\buinf|^2\rho_{\alpha 0}(\bx)\d\bx \cdot e^{\displaystyle -2\rate t}, 
\end{equation}
at exponential rate, $\rate$, dictated by the spatial scale $D_\infty$
\[
\rate = \zetaM \lambda_2(\DelMP_\infty) \gtrsim \frac{\zetaM}{(1+D_\infty)^\theta},
\qquad \Phi_\infty:=\Phi(D_\infty), \ \ \zetaM:=1-\frac{\max_\al M_\al}{\sum_\al M_\al}>0.
\]
\end{thm}

\noindent
The proof of theorem \ref{THM_1}, carried out in section \ref{sec:hydro_flocking} below, is achieved  by showing the decay of the fluctuations
\[
\left(\sum_{\al,\beta\in\mathcal{I}}\iint|\bu_\al(t,\bx)-\bu_\beta(t,\by)|^p\rho_\al(t,\bx)\rho_\beta(t,\by)\dbx\dby\right)^{1/p} \stackrel{t\rightarrow \infty}{\longrightarrow} 0.
\]
In particular, the  decay of the \emph{energy fluctuations}, corresponding to $p=2$,
\[
\delta E(t)=\sum_{\al,\bet\in {\mathcal I}}\iint |\bu_\al(t,\bx)-\bu_\bet(t,\by)|^2\rho_\al(t,\bx)\rho_\bet(t,\by)\dbx\dby,
\]
 and the decay of \emph{uniform fluctuations}, corresponding to $p=\infty$, 
\[
\delV(\bu(t))=\max_{\al,\bet\in{\mathcal I}}\sup_{\bx,\by\in{\mathcal S}(t)}
|\bu_\al(t,\bx)-\bu_\bet(t,\by)|, \qquad {\mathcal S}(t)= \cup_\al\text{supp}\{\rho_\al(t,\cdot)\},
\]
 imply that the whole crowd of different species remains  within a uniformly bounded finite diameter,  
 $D_\infty \leq D_0+C_\theta \cdot\delV_0 < \infty$ (with $C_\theta \lesssim 
 (1-\theta)^{\frac{\theta}{1-\theta}}$; consult \eqref{eq:Dbd} below). It follows that the fluctuations, $\delE(t), \delV(t)$, decay at exponential rate and that  all species `aggregate' around a limiting velocity $\mathbf{u}_\infty$.
Since the total mass    $M(t)=\sum_{\al} \int \rho_\al(t,\bx) \dbx$ and the total momentum $\bm(t)=\sum_{\al} \int \rho_\al\bu_\al(t,\bx) \dbx$
 are conserved in time, $M(t)=\Min$ and ${\mathbf m}(t)={\mathbf m}_0$, it follows that the  different species  flock together with the only possible limiting velocity $\displaystyle \bu_\al(t,\cdot) \stackrel{t\rightarrow \infty}{\longrightarrow} \buinf :=\frac{\mathbf{m_0}}{\Min}$.

\begin{rmk}[{\bf Why weighted Laplacian?}]
In case of equi-weighted species $M_\al\equiv 1$, the weighted Laplacian \eqref{eq:PhiLap} amounts to the usual  graph Laplacian $\Delta\Phi(r)$.
 Its Fiedler number, $\lambda_2(\DelPh(r))$,  quantifies the connectivity of the graph associated with the adjacency matrix $\Phi(r)$, \cite{Fiedler75}, \cite[proposition 6.1]{Mohar91}. Here, we advocate the use of the weighted graph Laplacian, 
$\DelMP(r)$, whose properties  are outlined in section \ref{sec:wFiedler} below; in particular, \emph{if} the number of species is finite, $|{\mathcal I}|<\infty$, then there holds, consult  \eqref{eq:compare},
\begin{equation}\label{eq:xcompare}
\frac{\Min}{\kappa^{2}|{\mathcal I}|} \leq \frac{\lambda_2(\DelMP)}{\lambda_2(\DelPh)}  \leq \frac{\Min\kappa^{2}}{|{\mathcal I}|}, \qquad \kappa=\frac{\max M_\alpha}{\min M_\alpha}, \quad \Min:=\sum_{\al\in {\mathcal I}} M_\al,
\end{equation}
and hence    $\Phi(r)$ is connected as long as  $\lambda_2\big(\DelMP(r)\big) \approx_\scrM \lambda_2\big(\DelPh(r)\big)>0$.  The advantage of using the weighted $\lambda_2(\DelMP(r))$, however,  is that it provides the right scaling  for the decay rate of multi-species dynamics \eqref{eq:revisit}, \textup{(}i\textup{)} independent of the condition number, $\kappa$, and \textup{(}ii\textup{)} independent of the number of different species, $|{\mathcal I}|$. On the other hand, if we accept $\kappa,|{\mathcal I}|$-dependence, then \eqref{eq:xcompare} implies that for \eqref{eq:Pareto} to hold it suffices to verify the Pareto `fat-tail' connectivity condition $\lambda_2(\DelPh(r)) \, {\gtrsim}_{{}_{\scrM,\kappa,|{\mathcal I}|}} (1+r)^{-\theta}$ with $\theta<1$.
\end{rmk}

\begin{rmk}[{\bf Game of alignment}]\label{rmk:game}
The graph Laplacian of the communication array $\Phi(r)$ is independent of the self-interacting kernels $\{\phi_{\al\al}\, | \,\al\in{\mathcal I}\}$. Thus, according to theorem \ref{THM_1}, flocking can be viewed as the outcome of a `game' in which agents from one species interact with different species  but are \emph{independent} of the interaction with their own kind. Alignment dynamics based on a game within a  single species was recently studied in \cite{griffin2019consensus}; a two-species ensemble dynamics in \cite{ha2017emergent}. A main feature in our multi-species alignment game (of two or more species) is that one can \emph{ignore interactions with its own kind}, i.e., set $\phi_{\al\al}=0$ in \eqref{Hydrodynamic_Flocking_eqs}  and yet the information will eventually be reflected through interactions with the other connected species leading to overall flocking. 
\end{rmk}

\begin{exm}
Consider the case of two species with $2\times 2$ symmetric communication array,
\[
\Phi= \left[
   \begin{array}{cc}
    0 & \phi_{12}(r)  \\ \\
    \phi_{21}(r) & 0  \\
   \end{array}
 \right], \qquad   \phi_{12}(r)=\phi_{21}(r) \gtrsim \frac{1}{(1+r)^\theta}, \ \  \theta<1.
 \]
In this case, agents in each of the two groups interact with the other group but \underline{not} with their own kind ($\phi_{11}=\phi_{22}\equiv 0$). The large-time behavior of such `game' leads to flocking.\newline
 Similarity, consider the case of four species with $4\times 4$ symmetric  communication array
\[
  \Phi=  \left[\begin{array}{cccc}
    0  & \phi_{12} & 0 & \phi_{14} \\
    \phi_{21} & 0 & \phi_{23} & 0  \\
     0 & \phi_{32} & 0 & \phi_{34}\\
    \phi_{41} & 0 & \phi_{43} & 0\\
   \end{array} \right], \quad \phi_{\al\bet}(r) =\phi_{\bet\al}(r)\gtrsim{(1+r)^{-\mu\cdot\min\{\al,\bet\}}}, \ \ \mu<1/3.
\]
Again, species do not interact with their own kind, but the connectivity of  inter-group interactions is strong enough to induce flocking.
\end{exm}

We close this section by noting  that the flocking of multi-species hydrodynamics \eqref{Hydrodynamic_Flocking_eqs} infers similar behavior of the underlying discrete multi-species Cucker-Smale dynamics
\[
\left\{
\begin{split}
& \ \ \dot{\bx}^i_\al=\bv_\al^i,\\
& \ \ \dot{\bv}_\al^i=\sum_{\beta\in \mathcal{I}}\frac{1}{N_\beta}\sum_{j=1}^{N_\beta} \phi_{\alpha\beta}(\tc|\bx^j_\beta-\bx^i_\al|)(\bv^j_\beta-\bv^i_\al),
\end{split}
\right.
 \qquad i\in 1,2,..., N_\al, \quad \al\in\mathcal I.
\]
The key feature is, again, weighted connectivity.
Thus, if the communication array $\Phi(r)=\{\phi_{\al\bet}(r)\}_{\al,\bet\in{\mathcal I}}$ satisfies the corresponding Pareto-type `fat-tail' connectivity condition\newline
$\lambda_2(\DelNP(r)) \gtrsim (1+r)^{-\theta}$
(weighted by the sizes of different species $\scrN:=\{N_\al\}_{\al\in{\mathcal I}}$),
 then the diameter of the different species remains bounded
depending on $1-\theta, \sum_\al N_\al$ and  $\delta\bv_0$,
\[
\max_{\al,\bet} \max_{i,j} |\bx_\al^i-\bx_\bet^j| \leq D_\infty< \infty,
\qquad \delta\bv_0:=\max_{i,j}\max_{\al,\bet}|\bv_\al^i(0)-\bv^j_\bet(0)|,
\]
 and the different species  flock towards a limiting velocity  $\bvinf$, 
\begin{equation}\label{eq:revisit}
\sum_{\al\in\mathcal{I}}|\bv^i_\alpha(t)-\bvinf|^2 \leq \sum_{\al\in\mathcal{I}} |\bv^i_{\alpha}(0)-\bvinf|^2 \cdot e^{\displaystyle -2\rate t}, 
\end{equation}
with exponential rate, $\rate$, dictated by the spatial scale $D_\infty$. 
The relation between connectivity and flocking was motivated by our earlier study  of   flocking for  discrete dynamics of one species, $\{(\bx^i(t), \bv^i(t))\}_{i=1}^N$, governed by  $\dot{\bv}^i=\frac{1}{N}\sum_{j=1}^N \phi(|\bx^j-\bx^i|)(\bv^j-\bv^i)$
and subject to \emph{short-range interactions}, \cite[Theorem 2.11]{MotschTadmor14}. It was shown that if connectivity persists in time  so that
\[
\int^\infty\lambda_2(\Delta {\Phi}(t))\dt=\infty, \qquad \Phi_{ij}(t)=\{\phi(|\bx_i(t)-\bx_j(t)|)\},
\]
then flocking follows, $\bv^i(t) \stackrel{t\rightarrow \infty}{\longrightarrow} \bv_\infty$.
%%%%%%%%%%%%%%%%%%%%%%%%%
\subsection{One- and two-dimensional smoothness --- sub-critical data}
%%%%%%%%%%%%%%%%%%%%%%%%%
The conditional statement that `smooth solutions must flock'  raises the question whether the  multi-species dynamics \eqref{Hydrodynamic_Flocking_eqs} admits global  smooth solutions. 

The case of one species was studied in one- and two-spatial dimensions. The one-dimensional well-posedness theory  \cite{CCTT16} provided precise characterization of global smooth solutions with sub-critical initial data, $u'_0+\phi*\rho_0\geq 0$. Global smoothness in two dimensions was proved for sub-critical initial data outlined in \cite{TanTadmor,HeTadmor17}. Here we develop the corresponding well-posedness of multi-species dynamics \eqref{Hydrodynamic_Flocking_eqs} in one- and two-spatial dimensions.

\medskip\noindent
 The one-dimensional  result is stated for  \emph{non-vacuous} initial data
 in the 1D torus.
\begin{thm}[{\bf Existence of smooth solutions --- one-dimensional dynamics}]\label{Thm_2}
Consider the multi-species dynamics \eqref{Hydrodynamic_Flocking_eqs}  subject to non-vacuous initial data $\{(\rho_{\al0}>0, {u}_{\al0})\}\in (L^\infty\cap L_+^1(\rt))\times W^{1,\infty}(\rt)$. If the initial condition satisfies the sub-critical threshold condition
\begin{equation}
 {u}'_{\al 0}(x)+\sum_{\beta\in\mathcal{I}} \phi_{\al\beta}*\rho_{\beta 0}(x)\geq 0,\quad \forall x\in\rt,\al\in \mathcal{I},
\end{equation}
then the multi-species dynamics \eqref{Hydrodynamic_Flocking_eqs} admits  global non-vacuous smooth solution, $(\rho_\al,\bu_\al)\in C(\rr_+; L^\infty\cap L^1(\rt))\times C(\rr_+;\dot W^ {1,\infty}(\rt))$.
\end{thm}

Turning to the two-dimensional case, we let $(\rho_\al,\bu_\al)$ be a solution of the 2D multi-species dynamics \eqref{Hydrodynamic_Flocking_eqs}. Global smoothness for sub-critical initial data is quantified in terms of the \emph{spectral gap} associated with the (symmetric part of the) $2\times 2$ velocity gradient matrix e.g., \cite{HeTadmor17}
\[
S_\al(t,\bx):=\frac{1}{2}\Big(\na \bu_\al(t,\bx)+(\na \bu_\al(t,\bx))^\top\Big), \quad (\na \bu_\al)_{ij}=\pa_j \bu^i_\al(t,\cdot),\enskip {i,j\in\{1,2\}}.
\]

\begin{thm}[{\bf Existence of smooth solutions --- two-dimensional dynamics}]\label{Thm_3}
Consider the two-dimensional multi-species dynamics \eqref{Hydrodynamic_Flocking_eqs} subject to compactly supported initial conditions $\{(\rho_{\al0}, \bu_{\al0})\}_{\al\in \mathcal{I}} \in (L^\infty\cap L_+^1(\rr^2))\times W^{1,\infty}(\rr^2)$. 
Assume a connected  communication array $\Phi(r)=\{\phi_{\al\beta}(r)\}_{\al,\bet\in{\mathcal I}}$ satisfying the `fat-tail' decay \eqref{eq:Pareto}, 
$\lambda_2(\DelMP(r)) \gtrsim (1+r)^{-\theta},\ \theta <1$.
There exists a constant $C_1=C_1(|\phi'_{\al \beta}|_{\infty},{\Min}, \gamma)$
(specified in \eqref{eq:setC0} below), such that if the initial fluctuations are not too large, $\delV_0\leq C_1$, and the following critical threshold conditions hold
\begin{subequations}
\begin{align}
\diver{\bu_{\al 0}}&(\bx)+\sum_{\beta\in\mathcal{I}}\phi_{\al\beta}*\rho_{\al0}(\bx) >0,\quad \forall\bx \in\rr^2,\label{eq:eCT}\\
\max_{\bx,\al} & |  \lambda_2(S_\al(0,\bx))-\lambda_1(S_\al(0,\bx))| <\frac{1}{2} C_1,  \quad  \label{eq:etaCT}
\end{align}
\end{subequations}
then the multi-species dynamics \eqref{Hydrodynamic_Flocking_eqs} admits a global smooth solution $(\rho_\al,\bu_\al)\in C(\rr_+; L^\infty\cap L^1(\rr^2))\times C(\rr_+;\dot W^ {1,\infty}(\rr^2))$ with large time hydrodynamic flocking behavior 
$\bu_\al(t,\bx) \rightarrow \buinf$. 
\end{thm}
%%%%%%%%%%%%%%%%%%%%%%%
\subsection{Multi-species aggregation model}
%%%%%%%%%%%%%%%%%%%%%
We turn our attention to the multi-species  \emph{aggregation} dynamics. The aggregation dynamics of a  single-species  arises in different  contexts of modeling opinion dynamics, the rendezvous problem, etc; see e.g.,  \cite{BertozziCarrilloLaurent09,HuangBertozzi10,FetecauHuangKolokolnikov11,
CarrilloDiFrancescoFigalliLaurentSlepcev11,Grindrod88,Poupaud02} and the reference therein,
\[
\left\{\begin{split}
\partial_t \rho -\na\cdot \big( \big((\bx\phi )*\rho\big)\rho\big)&=0\\
 \rho(t=0,\bx)&=\rho_0(\bx),
 \end{split}\right. \qquad \forall\bx\in\rr^d.
\]
Global smooth solutions  tend to a Dirac mass which concentrates at the invariant center of mass. This large time \emph{concentration} reflects the emergence of consensus (in opinion dynamics) and rendezvous problem (in distributed sensor-based dynamics) etc.
There is also an increasing interest in two species-aggregation models, \cite{Grindrod88} and the recent works  \cite{FrancescoFagioli2013,KazianouLiaoVauchelet17}, and \cite{EversFetecauKolokolnikov17}. In particular,  \cite{FrancescoFagioli2013,KazianouLiaoVauchelet17} study  1D measure-valued solutions of the 2-species dynamics after blow-up  in the special case of $\phi_{\al\beta}\equiv \phi$, and  \cite{EversFetecauKolokolnikov17} categorize the possible steady states of the two-species system.
Here we  extend the discussion to  the multi-species setting
\begin{equation}\label{EQ:AggregationEq_multi-groups}
\left\{
\begin{split}
\qquad\pa_t \rho_\al- \sum_{\beta\in\mathcal{I}} \na\cdot ((\bx\phi_{\al\beta} )*\rho_\beta) \rho_\al)&=0, \\
\rho_\al(t=0,\bx)&=\rho_{\al 0}(\bx), 
\end{split} \right. \qquad \forall\bx\in \rr^d, \al\in \mathcal{I}.
\end{equation}
The different species  are identified by their densities --- $\rho_\al$ denotes the agent density in the species $\al$, a macroscopic realization of the agent-based dynamics of a species with $N_\al$ agents, each has position, $\bx_\al^i$, and interacts with the other species
\[
\dot \bx_\al^i=-\sum_{\bet\in{\mathcal I}}\frac{1}{N_\bet}\sum_{j=1}^{N_\bet} \phi_{\al\bet}(|\bx_\al^i-\bx_\bet^j|)(\bx_\bet^j-\bx_\al^i).
\]

In this paper, we extend the results to the multi-species setting and give explicit sufficient condition to guarantee consensus under the assumption that the communication array  $\Phi=\{\phi_{\al\beta}\}$ form a  connected network. Our main theorem is summarized in the following.

\begin{thm}[{\bf First-order aggregation}]\label{Thm_4}
Let $\{\rho_\al(t,\cdot)\} \in W^1_+(\rr^d)$ be a strong solution of  the multi-species aggregation system \eqref{EQ:AggregationEq_multi-groups} subject to compactly supported initial data $(\rho_{\al0})_{\al\in\mathcal{I}}$ with a finite diameter
\[
D_0= \sup_{\bx,\by\in{\mathcal S}_0}|\bx-\by|,\qquad {\mathcal S}_0=\cup_\al \textrm{supp}\, \{\rho_{\al 0}\}
\]
 and governed by  radially symmetric decreasing kernels $\{\phi_{\al\bet}(r)\}$
 \eqref{eq:symm}. Let $\Phi_0$ denote the communication array scaled at the initial diameter, $\Phi_0=\{\phi_{\al\bet}(D_0)\}_{\al,\bet\in{\mathcal I}}$. There holds
 \[
 \delD(t) \leq \delD(0)\cdot e^{-\displaystyle 2\zetaM\lambda_2(\DelMP_0)t} , \quad \delD(t):=\sum_{\al,\bet\in{\mathcal I}}\iint |\bx-\by|^2\rho_\al(t,\bx)\rho_\bet(t,\by)\dbx\dby
 \]
In particular, if the communication array $\Phi_0$ is  connected, then the different species $\{\rho_\alpha\}_{\al\in {\mathcal I}}$ aggregate towards the limiting position  $\overline{\bx}_\infty$ 
\begin{equation}\label{eq:xrevisit}
\sum_\al\int |\bx-\overline{\bx}_\infty|^2\rho_\al(t,\bx)\dbx \lesssim \sum_\al\int |\bx-\overline{\bx}_\infty|^2\rho_{\al 0}(\bx)\dbx \cdot e^{\displaystyle - 2\rate t}, 
\end{equation}
at exponential rate, $\rate$, dictated by the initial spatial scale $D_0$,
\[
\rate=\zetaM\lambda_2(\DelMP_0), \qquad \Phi_0=\{\phi_{\al\bet}(D_0)\}, \quad \zetaM=1-\frac{\max_\al M_\al}{\sum_\al M_\al}>0.
\]
\end{thm}

\begin{rmk}
The proof of theorem \ref{Thm_4}, carried out in section \ref{sec:agg} below, 
 implies that if the communication array $\{\phi_{\al\bet}(D_0)\}$ forms a connected array then all species `aggregate' around a limiting position $\overline{\bx}_\infty$.
Since the center of mass $\displaystyle \frac{1}{M(t)}\sum_\al \int \rho_\al(t,\bx)\bx\dbx$ is conserved in time, it follows that the  different species  aggregate around $\overline{\bx}_\infty=$ center of mass as the only possible limiting position.
As before, aggregation depends on path connectivity but are independent of the  self-interacting kernels, $\{\phi_{\al\al} \, | \, \al\in{\mathcal I}\}$ which are allowed to vanish.
\end{rmk}
\begin{rmk}[Existence of smooth solution] Assume that $\bx\phi_{\al\beta}\in W^{1,\infty}(\rr^d)$. Then, the multi-species dynamics which we rewrite as
\[
\pa_t \rho_\al +\sum_{\beta\in\mathcal{I}}((\bx\phi_{\al\beta} )*\rho_\beta)\cdot \na \rho_\al=-\sum_{\beta\in\mathcal{I}}\na\cdot ((\bx\phi_{\al\beta} )*\rho_\beta)\rho_\al
\]
implies the uniform bound
\[
\frac{d}{dt}|\rho_\al|_\infty\leq \sum_{\beta\in\mathcal{I}}|\na\cdot(\bx\phi_{\al\beta} )*\rho_\beta|_\infty|\rho_\al|_\infty \lesssim \sum_{\beta\in\mathcal{I}}|\na\cdot(\bx \phi_{\al\beta})|_\infty M_\beta |\rho_\al|_\infty.
\]
The uniform bound of the $\rho_\al$'s implies higher $H^s$ Sobolev bounds by standard energy estimates. Thus, for example we have the $H^1$-bound
\begin{align*}
\ddt\sum_{\al\in\mathcal{I}}|\na \rho_\al|_2^2=&\sum_{\al,\beta\in{\mathcal I}}\int |\na \rho_\al\na (\na\cdot(\phi\mathbf{x})*\rho_\beta \rho_\al+\phi\mathbf{x}*\rho_\beta\cdot\na \rho_\al)|\dbx\\
\leq &\sum_{\al,\beta\in{\mathcal I}}\left(|\na \rho_\al|_2| \na (\phi \mathbf{x})|_\infty|\na \rho_\beta|_2||\rho_\al|_\infty+3|\na \rho_\al|_2^2|\na(\phi\mathbf{x})|_\infty|\rho_\beta|_1\right) \\
\lesssim & \sum_{\al\in\mathcal{I}}|\na \rho_\al|_2^2.
\end{align*}
\end{rmk}

The paper is organized as follows: In section \ref{sec:hydrodynamics}, we  formally derive the macroscopic model \eqref{Hydrodynamic_Flocking_eqs} as the large-crowd dynamic description of the discrete agent-based model. In section \ref{sec:wFiedler} we prepare  the weighted Poincar\'e inequality associated with weighted graph Laplacian which will be used in the sequel. In section \ref{sec:hydro_flocking}, we prove the main results of flocking: decay of energy fluctuations in theorem \ref{THM_4_1} and decay of uniform fluctuations in \ref{THM_4_2}, which in turn lead to the proof of theorem \ref{THM_1}. In section \ref{sec:global}, we prove the existence of global smooth solutions --- the one- and two-dimensional setup in   theorem \ref{Thm_2} and respectively \ref{Thm_3}. Finally in section \ref{sec:agg}, we treat the multi-species aggregation of system, proving Theorem \ref{Thm_4}.

%%%%%%%%%%%%%%%%%%%%%%%%%%%%%%%
\section{Derivation of the mesoscopic and hydrodynamic models}\label{sec:hydrodynamics}
%%%%%%%%%%%%%%%%%%%%%%%%%%%%%%%%%%%

In this section, we formally derive the multi-species hydrodynamics \eqref{Hydrodynamic_Flocking_eqs} from the  underlying multi-species agent-based dynamics. 
To this end, we first derive a mesoscopic Vlasov type description  which in turn yields   the macroscopic description \eqref{Hydrodynamic_Flocking_eqs}. 

To formulate the mesoscopic equation, we first define the following empirical probability measure associated to the species $\al$, which represents the probability of finding an agent from species $\al$ at position $\bx$ with velocity $\bv$:
\begin{align}
f_\al(t,\bx,\bv)=\frac{1}{N_\al}\sum_{i=1}^{N_\al}\delta_{\bx^i_\al(t)}\otimes \delta_{\bv^i_\al(t)}.\label{empirical_measure_x }%\\
\end{align}
Here $N_\al$ denotes the number of agents in the group $\al$. Evolution of each probability density $f_\al$ can be derived by testing $\pa_t f_\al$ against an arbitrary smooth function $\test$ through equation \eqref{Hydrodynamic_Flocking_eqs} 
\begin{align}
\iint\pa_t f_\al(t,\bx,\bv)\test(\bx,\bv)\dbx\dbv=&\frac{1}{N_\al}\sum_{i=1}^{N_\al}\pa_t\test(\bx^i_\al(t),\bv^i_\al(t))\nonumber\\
=&\frac{1}{N_\al}\sum_{i=1}^{N_\al}[\dot{\bx}^i_\al\cdot \na_\bx\test(\bx_\al^i(t),\bv^i_\al(t))+\dot{\bv}^i_\al\cdot\na_\bv\test(\bx^i_\al,\bv^i_\al)]\label{mid_step_1}\\
=&\frac{1}{N_\al}\sum_{i=1}^{N_\al}[{\bv}^i_\al\cdot \na_\bx\test(\bx_\al^i,\bv^i_\al)+{F}^i_\al\cdot\na_\bv\test(\bx^i_\al,\bv^i_\al)],\nonumber
\end{align}
with an alignment forcing $F_\al^i$  given by
\[
\begin{split}
F^i_\al=\sum_{\beta\in\mathcal{I}}\frac{1}{N_\beta}\sum_{j=1}^{N_\beta}\phi_{\al\beta}(|\bx^j_\beta-\bx^i_\al|)(\bv^j_\beta-\bv^i_\al)
= \sum_{\bet \in {\mathcal I}}L_{\al\beta}(f_\beta)(\bx^i_\al,\bv^i_\al),
\end{split}
\]
where $\displaystyle L_{\al\beta}(f_\beta)(\bx^i_\al,\bv^i_\al) :=\iint\phi_{\al\beta}(|\by-\bx^i_\al|)(\bw-\bv^i_\al)f_\beta(\by,\bw)\dby\dbw$.
Formal integration by parts in \eqref{mid_step_1} yields
\begin{align*}
\iint\pa_t &f_\al(t,\bx,\bv)\test(\bx,\bv)\dbx\dbv\\
\qquad =&\iint[\bv\cdot\na_\bx\test(\bx,\bv)+\sum_{\beta\in\mathcal{I}} L_{\al\beta}(f_\beta)(\bx,\bv)\cdot\na_\bv\test(\bx,\bv)]f_\al(\bx,\bv)\dbx\dbv\\
\qquad =&-\iint\left[\bv\cdot\na_\bx f_\al(\bx,\bv)+\na_\bv\cdot\left(\sum_{\beta\in\mathcal{I}}  L_{\al\beta}(f_\beta)f_\al\right)\right]\test \dbx\dbv.
\end{align*}
Since the test function $\test$ is arbitrary, the above integral equation yields the  \emph{mesoscopic scale equation}
\begin{eqnarray}\label{mesoscopic_model}
\pa_t f_\al(\bx,\bv)+\bv\cdot\na_\bx f_\al(\bx,\bv)+\na_\bv\cdot\left(\sum_{\beta\in\mathcal{I}} L_{\al\beta}(f_\beta)f_\al\right)=0.
\end{eqnarray}
The bi-linear expression inside the parenthesis on the left represents the inter-species alignment interactions.
This completes the derivation from the microscopic agent-based dynamics to the mesoscopic scale dynamics.

The  hydrodynamic description is formally achieved by calculating the time evolution of the `observable moments', e.g., the mass density and the momentum density:
\begin{equation}
\left\{
\begin{split}
\quad \rho_\al(t,\bx):=&\int_{\rr^d} f_\al(t,\bx,\bv)\dbv;\\
\rho_\al \bu_\al(t,\bx):=&\int_{\rr^d} \bv f_\al(t,\bx,\bv)\dbv.\label{Momentum_density}
\end{split}\right.
\end{equation}
By integrating the mesoscopic equation \eqref{mesoscopic_model} in the velocity variable $\bv$ and applying integration by parts, we derive the mass equation  for $\rho_\al$:
\begin{align}\label{Hydrodynamic_scale_continuity_equation}
(\rho_\al)_t+\na_\bx\cdot(\rho_\al \bu_\al)=0,\quad \forall\al \in\mathcal{I}.
\end{align}
The dynamics of the momentum $\rho_\al\bu_\al$ is obtained by integrating  \eqref{mesoscopic_model} against $\bv$,
\begin{align}\label{pre_hydrodynamic_velocity_equation}
0=\int\left[\pa_t(\bv f_\al)+\bv(\bv\cdot\na_\bx f_\al)+\bv\na_\bv\cdot\left(\sum_{\beta\in\mathcal{I}} L_{\al\beta}(f_\beta)f_\al\right)\right]\dbv=:I+II+III.
\end{align}
The first term is the time derivative of the momentum density, $\rho_\al \bu_\al$ in \eqref{Momentum_density},
\begin{align}
I=\pa_t(\rho_\al \bu_\al);\label{I}
\end{align}
the second term $II$ can be rewritten as
\begin{align}\label{II_and_defn_P}
\begin{split}
II & =\na_\bx\cdot(\rho_\al \bu_\al\otimes \bu_\al)+\na_\bx\cdot {\int (\bu_\al-\bv)\otimes (\bu_\al-\bv)f_\al(\bx,\bv_\al)\dbv}\\
 &=:\na_\bx\cdot(\rho_\al \bu_\al\otimes \bu_\al)+\na_\bx \cdot P_\al,
 \end{split}
\end{align}
where  $P_\al$ is interpreted as pressure tensor. 
For the third term $III$ in \eqref{pre_hydrodynamic_velocity_equation}, we use integration by parts to rewrite it as follows
\begin{align}
III=&\int \bv\na_\bv\cdot\left(\sum_{\beta\in\mathcal{I}} L_{\al\beta}(f_\beta)f_\al\right)\dbv
=-\sum_{\beta\in\mathcal{I}} \int L_{\al\beta}(f_\beta)f_\al \dbv\nonumber\\
=&-\sum_{\beta\in\mathcal{I}}\iiint \phi_{\al\beta}(|\by-\bx|)(\bw-\bv)f_\beta(\by,\bw)f_\al(\bx,\bv)\dby\dbw\dbv\nonumber\\
=&-\sum_{\beta\in\mathcal{I}}\iiint \phi_{\al\beta}(|\by-\bx|)(\bw f_\beta(\by,\bw))f_\al(\bx,\bv)\dbw\dby\dbv \nonumber\\
 & \qquad  +\sum_{\beta\in\mathcal{I}}\iiint \phi_{\al\beta}(|\by-\bx|)f_\beta(\by,\bw)(\bv f_\al(\bx,\bv))\dbv\dby\dbw\label{III}\\
=&-\sum_{\beta\in\mathcal{I}}\iint \phi_{\al\beta}(|\by-\bx|)(\rho_\beta \bu_\beta)(\by)f_\al(\bx,\bv)\dby\dbv \nonumber\\
 & \qquad  +\sum_{\beta\in\mathcal{I}}\iint \phi_{\al\beta}(|\by-\bx|)f_\beta(\by,\bw)(\rho_\al \bu_\al)(\bx)\dby\dbw\nonumber\\
=& -\sum_\bet\int\phi_{\alpha\beta}(\tc|\bx-\by|)(\bu_\beta(\by)-\bu_\alpha(\bx))\rho_\alpha(\bx)\rho_\beta(\by)\dby.\nonumber
\end{align}
Now combining \eqref{I}, \eqref{II_and_defn_P} and  \eqref{III}  we obtain the \emph{hydrodynamic momentum equation} 
\begin{align*}
\pa_t (\rho_\al \bu_\al)&+\na \cdot (\rho_\al \bu_\al\otimes \bu_\al)+\na_\bx\cdot P_\al=\sum_\bet\int\phi_{\alpha\beta}(\tc|\bx-\by|)(\bu_\beta(\by)-\bu_\alpha(\bx))\rho_\alpha(\bx)\rho_\beta(\by)\dby.
\end{align*}
Similar to the one-species (hydro-)dynamics, \cite{HaTadmor08, KangFigalli17}, we limit ourselves to the \emph{mono-kinetic ansatz} $f_\al(\bx,\bv)=\rho_\al(\bx)\delta_{\bu_\al(\bx)}(\bv)$ to impose the pressure closure $P_\al\equiv 0$, and end up with the multi-species   hydrodynamics  \eqref{Hydrodynamic_Flocking_eqs}.

%%%%%%%%%%%%%%%%%%
\section{Weighted Poincar\'e inequalities}\label{sec:wFiedler}
%%%%%%%%%%%%%%%%%
Given an $N\times N$ symmetric array $\APhi=\{\aphi_{\al\bet}\}$ of non-negative entries, and  positive weights $\scrW:=\{\wrho_\al\}$, we are concerned with a \emph{weighted} Poincar\'e inequality  of the form
\begin{equation}\label{eq:weighted-Poi}
 \sum_{\alpha,\beta} \aphi_{\alpha\beta}|\ux_\alpha-\ux_\beta|^2\wrho_\alpha\wrho_\beta \geq \rate \sum_{\alpha,\beta}|\ux_\alpha-\ux_\beta|^2\wrho_\alpha\wrho_\beta, \qquad \rate>0.
\end{equation}
The standard Poincar\'e (or Courant-Fisher) inequality tells us that, in case of equal weights $\wrho_\alpha\equiv1$, \eqref{eq:weighted-Poi} holds with optimal $\rate$ given by  the Fielder number,
$\rate=\lambda_2(\Delta \APhi)/N$, where $\Delta \APhi$ is the  graph Laplacian, \cite{Fiedler75}, \cite[proposition 6.1]{Mohar91},
\begin{equation}\label{eq:unwPoi}
 \sum_{\alpha,\beta} \aphi_{\alpha\beta}|\ux_\alpha-\ux_\beta|^2 \geq \frac{\lambda_2(\Delta \APhi)}{N} \sum_{\alpha,\beta}|\ux_\alpha-\ux_\beta|^2. \qquad (\Delta \APhi)_{\alpha\beta}:= -(1-\delta_{\alpha\beta})\aphi_{\alpha\beta}+\delta_{\alpha\beta}\sum_{\gamma\neq \al} \aphi_{\alpha\gamma}.
\end{equation}
To treat the case of general weights, we let $\DelW \APhi$ denote the \emph{weighted} Laplacian 
\begin{equation}\label{eq:Lap}
\left(\DelW\APhi\right)_{\alpha\beta}=\left\{\begin{array}{ll} -\aphi_{\alpha\beta}\sqrt{\wrho_\alpha\wrho_\beta}, & \alpha\neq \beta,\\ \\
\displaystyle \sum_{\gamma\neq \alpha} \aphi_{\alpha\gamma}\wrho_\gamma, & \alpha=\beta.
\end{array}\right.
\end{equation}
Observe that $\DelW \APhi$ is symmetric yet \emph{not} row stochastic.
Its second eigenvalue dictates the following weighted Poincar\'{e} inequality for arbitrary $N$-vectors $\bux=\{\ux_\al\}$.
\begin{lem}[{\bf Weighted Poincar\'e inequality -- vectors}]\label{lem:Poin}
There holds 
\begin{equation}\label{eq:Poin}
\sum_{\alpha,\beta} \aphi_{\alpha\beta}|\ux_\alpha-\ux_\beta|^2\wrho_\alpha\wrho_\beta \geq \frac{\lambda_2(\DelW \APhi)}{\sum_\beta \wrho_\beta}\sum_{\alpha,\beta}|\ux_\alpha-\ux_\beta|^2\wrho_\alpha\wrho_\beta, \quad  \scrW:=\{\wrho_\al\}.
\end{equation}
\end{lem}

\noindent
\begin{rmk}[{\bf Scaling}]\label{rem:scaling} Lemma \ref{lem:Poin}  with $\wrho_\al \equiv 1$ recovers the regular Poincar\'e inequality \eqref{eq:unwPoi}. Observe that \eqref{eq:unwPoi} together with the obvious  $\min \wrho_\al^2\leq \wrho_\alpha\wrho_\beta \leq \max \wrho_\al^2$ yield a desired bound \eqref{eq:weighted-Poi} with $\rate=\lambda_2(\Delta \APhi)/(\kappa^{2}N)$,
\begin{equation}\label{eq:notsharp}
\sum_{\alpha,\beta} \aphi_{\alpha\beta}|\ux_\alpha-\ux_\beta|^2\wrho_\alpha\wrho_\beta \geq \lambda_2(\Delta \APhi)\frac{1}{\kappa^2 N} \sum_{\alpha,\beta}|\ux_\alpha-\ux_\beta|^2\wrho_\alpha\wrho_\beta, \quad \kappa:=\frac{\max\wrho_\alpha}{\min\wrho_\alpha}.
\end{equation}
The point to note here is that this bound in terms of $\lambda_2(\Delta \APhi)$ depends on $N$ and the condition number $\kappa$. In contrast, the weighted bound \eqref{eq:Poin} which involves $\lambda_2(\DelW(\APhi))$ has the right `scaling', depending on the (usually invariant) total mass of the weights but otherwise it is independent $N,\kappa$. In particular,  the size of  $A=\{a_{\al\bet}\}$   is allowed to grow unboundedly large with $N$ as long as  the total weight remains finite, $\sum_{\bet} \wrho_\bet<\infty$.
\end{rmk}

\begin{proof}[{Proof of Lemma \ref{lem:Poin}}] The sum on the left of \eqref{eq:Poin} can be expressed as a bi-linear form in terms of the weighted Laplacian $\DelW\APhi$ in \eqref{eq:Lap} (here and below $\bw$ is the vector of weights $\bw=(\wrho_1,\wrho_2,\ldots)^\top$ and we abbreviate 
$\sqrt{\bw}\bux=(\sqrt{\wrho_1}\ux_1,\sqrt{\wrho_2}\ux_2,\ldots)^\top$)
\begin{equation}\label{eq:weighted_sum}
\begin{split}
\big\langle \left(\DelW\APhi\right)\sqrt{\bw}\bux,\sqrt{\bw}\bux\big\rangle 
&:= -\sum_\al\sum_{\bet\neq \al} \aphi_{\alpha\beta}\sqrt{\wrho_\alpha\wrho_\beta}\sqrt{\wrho_\alpha} \sqrt{\wrho_\beta}\,\ux_\alpha \ux_\beta
+\sum_\alpha\sum_{\beta\neq \alpha}\aphi_{\alpha\beta}\wrho_\beta \wrho_\alpha|\ux_\alpha|^2  \\
& \equiv \frac{1}{2} \sum_\al\sum_{\bet\neq \al} \aphi_{\alpha\beta}|\ux_\beta-\ux_\alpha|^2\wrho_\alpha\wrho_\beta, 
\end{split}
\end{equation}
which shows  that the symmetric Laplacian $\DelW\APhi $ is positive semi-definite with eigenvalues $0= \lambda_1 \leq \lambda_2 \leq \ldots$. Here, $\lambda_1$ is the zero eigenvalue associated with the eigenvector $\sqrt{\bw}:=(\sqrt{\wrho_1}, \sqrt{\wrho_2},\ldots)^\top$,
\[
\Big(\left(\DelW \APhi\right) \sqrt{\bw}\Big)_\alpha
= -\sum_{\beta\neq \alpha}\aphi_{\alpha\beta}\sqrt{\wrho_\alpha\wrho_\beta}\sqrt{\wrho_\beta} + \sum_{\beta\neq \alpha} \aphi_{\alpha\beta}\wrho_\beta\sqrt{\wrho_\alpha} \equiv 0,
\]
and hence $\left(\DelW \APhi\right) (\sqrt{\bw}\,\overline{\bux})=0$ for any constant vector $\overline{\bux}=\overline{\ux}(1,1,\ldots, 1)^\top$. In particular, for  
$\displaystyle \overline{\ux}=\frac{\sum_\beta \wrho_\beta \ux_\beta}{\sum_\beta \wrho_\beta}$ the orthogonal complement of $\sqrt{\bw}\, \overline{\bux}$ is given by $\{\sqrt{\bw}(\bux-\overline{\bux})\}$,
\[
\big\langle \sqrt{\bw}(\bux-\overline{\bux}), \sqrt{\bw}\,\overline{\bux}\big\rangle =0, \qquad
\overline{\ux}:=\frac{\sum_\beta \wrho_\beta \ux_\beta}{\sum_\beta \wrho_\beta},
\]
hence
\begin{equation}\label{eq:ortho} 
\begin{split}
\big\langle \left(\DelW\APhi\right)\sqrt{\bw}\bux,\sqrt{\bw}\bux\big\rangle
 &= \left\langle \left(\DelW\APhi\right)\sqrt{\bw}(\bux-\overline{\bux}),\sqrt{\bw}(\bux-\overline{\bux})\right\rangle \\
 & \geq  \lambda_2(\DelW\APhi) \times|\sqrt{\bw}(\bux-\overline{\bux})|^2.
 \end{split}
\end{equation}
A straightforward computation yields
\begin{align*}
|\sqrt{\bw}(\bux-\overline{\bux})|^2  = & \sum_\alpha \wrho_\alpha |\ux_\alpha|^2-2\sum_\alpha \wrho_\alpha \ux_\alpha \overline{\ux} + \sum_\alpha\wrho_\alpha|\overline{\ux}|^2 
=   \sum_\alpha \wrho_\alpha |\ux_\alpha|^2 -\frac{|\sum_\beta \wrho_\beta \ux_\beta|^2}{\sum_\beta\wrho_\beta}\\
= & \frac{1}{\sum_\beta \wrho_\beta}\left(\sum_{\alpha,\beta}\wrho_\alpha\wrho_\beta|\ux_\alpha|^2
-\sum_\beta \wrho_\beta^2|\ux_\beta|^2 -\sum_\al\sum_{\bet\neq \al}\wrho_\alpha\wrho_\beta\, \ux_\alpha\ux_\beta\right)\\
= & \frac{1}{2\sum_\beta \wrho_\beta}\left(\sum_\al\sum_{\bet\neq\al}\wrho_\alpha\wrho_\beta|\ux_\alpha|^2 + \sum_\al\sum_{\bet\neq\al}\wrho_\alpha\wrho_\beta|\ux_\beta|^2
-2\sum_\al\sum_{\bet\neq\al}\wrho_\alpha\wrho_\beta\, \ux_\alpha \ux_\beta\right)  \\
 \equiv &  \frac{1}{2\sum_\beta \wrho_\beta} \sum_\al\sum_{\bet\neq\al} |\ux_\alpha-\ux_\beta|^2\wrho_\alpha\wrho_\beta,
\end{align*}
and \eqref{eq:Poin} follows from \eqref{eq:weighted_sum} and \eqref{eq:ortho}.
\end{proof}
\begin{rmk}[{\bf Optimality}]\label{rem:compare} The proof of Lemma \ref{lem:Poin} shows   the optimality of the weighted Laplacian \textup{(} --- choose $\sqrt{\bw}\bx$ as the second, Fiedler eigenvector of $\DelW \APhi$\textup{)},
leading to a Courant-Fisher-type  characterization 
\begin{equation}\label{eq:CF}
\frac{\lambda_2(\DelW \APhi)}{\sum_\beta \wrho_\beta} = \min_{|\delta\bux|_\wrho=1} \sum_\al\sum_{\bet\neq\al} \aphi_{\alpha\beta}|\ux_\alpha-\ux_\beta|^2\wrho_\alpha\wrho_\beta, \qquad |\delta\bux|^2_{\wrho}:= \sum_\al\sum_{\bet\neq\al}|\ux_\alpha-\ux_\beta|^2\wrho_\alpha\wrho_\beta.
\end{equation}
Hence, comparing this  with \eqref{eq:notsharp} one concludes
\begin{equation}\label{eq:compare}
 \frac{1}{\kappa^{2}N}\lambda_2(\Delta \APhi) \leq \lambda_2(\DelW\APhi)\frac{1}{\sum_\bet \wrho_\bet}  \leq \frac{\kappa^{2}}{N}\lambda_2(\Delta \APhi), \qquad \kappa=\frac{\max\wrho_\alpha}{\min\wrho_\alpha}.
\end{equation}
\end{rmk}

The array $\APhi$   forms a connected graph if it has a positive Fiedler number, $\lambda_2\big(\DelW \APhi\big)>0$. 
In particular, $\APhi$ being a connected graph, the degree of its nodes are positive, $\displaystyle \sum_{\bet\neq \gamma} \aphi_{\gamma\bet}\wrho_\bet >0$. To quantify this statement which will be used below, we  appeal to \eqref{eq:Poin}
\[
\sum_\al\sum_{\bet\neq\al}{\aphi}_{\al\bet}|{\ux}_\al-{\ux}_\bet|^2\wrho_\al\wrho_\bet \geq
\frac{\lambda_2(\DelW \APhi)}{\sum_\bet \wrho_\bet}\sum_\al\sum_{\bet\neq\al}|{\ux}_\al-{\ux}_\bet|^2\wrho_\al\wrho_\bet.
\]
Fix an index $\gamma$ and test  the last inequality with the  vector $\left\{\bux\ \Big| \ \ux_\al=\left\{\begin{array}{cc}0 & \al\neq \gamma,\\ 
\rate, & \al=\gamma.\end{array}\right.\right\}$, with normalization factor $\displaystyle \rate=\Big(2\sum_{\bet\neq \gamma}\wrho_\bet\wrho_\gamma\Big)^{-1/2}$ so that $|\delta \bux|_\wrho=1$. The sum on the left  is reduced to the  $(\gamma,\bet)$-terms with $\beta\neq \gamma$,  for which $|\ux_{\gamma}-\ux_\bet|^2=\rate^2$ and $(\al,\gamma)$-terms with $\al\neq \gamma$ for which  $|\ux_\al-\ux_{\gamma}|^2=\rate^2$ and \eqref{eq:Poin} amounts to
$\displaystyle 2\rate^2\sum_{\bet\neq \gamma} \aphi_{\gamma\bet}\wrho_{\gamma} \wrho_\bet 
\geq \frac{\lambda_2(\DelW \APhi)}{\sum_\bet \wrho_\bet}$ and we conclude
\begin{equation}\label{eq:zeta}
\dzA{\gamma}:=\sum_{\bet\neq\gamma} \aphi_{\gamma\bet}\wrho_\bet \geq   \frac{\sum_{\bet\neq \gamma} \wrho_\bet}{\sum_\bet \wrho_\bet}\lambda_2(\DelW\APhi) \geq \zeta_\scrW\lambda_2(\DelW\APhi), \quad \zeta_\scrW=1-\frac{\max_\bet \wrho_\bet}{\sum_\bet \wrho_\bet}>0.
\end{equation}

Next, we extend Lemma \ref{lem:Poin} from vectors to vector-\emph{functions}, seeking an  inequality of the form
\[
\sum_{\alpha,\beta}\aphi_{\alpha\beta} \iint |u_\alpha(\bx)-u_\beta(\by)|^2\rho_\alpha(\bx)\rho_\beta(\by)\dbx\dby \geq \rate \sum_{\alpha,\beta} \iint |u_\alpha(\bx)-u_\beta(\by)|^2\rho_\alpha(\bx)\rho_\beta(\by)\dbx\dby.
\]
Clearly we can use $\rate=\min_{\al\bet} \aphi_{\alpha\beta}$. But there is a sharper threshold, $\rate=\rate_{\APhi}$, which allows some (-- and in fact most) of the entries $\{\aphi_{\al\bet}\}$ to vanish yet  $\rate_{\APhi}>0$. In particular, $\rate_{\APhi}$ is \emph{independent} of the (amplitudes of the) self-interacting terms $\{\aphi_{\alpha\alpha}\}$. 
\begin{lem}[{\bf Weighted Poincar\'e inequality -- vector-functions}]\label{lem:Poinv} Let $\{\wrho_\gamma\}$ be non-negative weight functions with positive  finite masses
$\displaystyle M_\gamma=\int \wrho_\gamma(\bx)\dbx>0$. 
There holds 
\begin{equation}\label{eq:alphaneqbeta}
\begin{split}
\sum_{\alpha\neq \beta} \aphi_{\alpha\beta} \iint |u_\alpha(\bx)-u_\beta(\by)|^2&\wrho_\alpha(\bx)\wrho_\beta(\by)\dbx\dby \\
 & \geq \rate \sum_{\alpha,\beta} \iint |u_\alpha(\bx)-u_\beta(\by)|^2\wrho_\alpha(\bx)\wrho_\beta(\by)\dbx\dby,
 \end{split}
 \end{equation}
 with $\rate=\rate_{\APhi}$ given by
 \[
 \rate_{\APhi} = \lambda_2(\DelMA)\frac{\zetaM}{M}, \qquad \zetaM=1-\frac{\max_\gamma M_\gamma}{M}, \quad M=\sum_\gamma M_\gamma.
 \]

\end{lem}
The  bound \eqref{eq:alphaneqbeta} is at the heart of matter: note that the  self-interacting terms $\sum_{\alpha} \iint |u_\alpha(\bx)-u_\alpha(\by)|^2\wrho_\alpha(\bx)\wrho_\alpha(\by)\dbx\dby$ are missing on its left  but present in the lower-bound on the right.
\begin{proof}[{Proof of Lemma \ref{lem:Poinv}}] Denote the average, $\displaystyle \overline{u}_\alpha:=\frac{\int \wrho_\alpha u_\alpha(\bx)\dbx}{\int \wrho_\alpha(\bx)\dbx}$. Since
 $\displaystyle \int_\bx \big(u_\al(\bx)-\overline{u}_\al\big)\wrho_\al(\bx)\dbx$ and $\displaystyle \int_\by \big(u_\bet(\by)-\overline{u}_\bet\big)\wrho_\bet(\by)\dby$ vanish, we can  decompose the integral on the left of \eqref{eq:alphaneqbeta}
\[
\begin{split}
\iint |u_\alpha(\bx)-u_\beta(\by)|^2&\wrho_\alpha(\bx)\wrho_\beta(\by)\dbx\dby \\
& \equiv \iint \Big(|u_\alpha(\bx)-\overline{u}_\alpha|^2 + |\overline{u}_\alpha-\overline{u}_\beta|^2 + |\overline{u}_\beta-u_\beta(\by)|^2\Big)\wrho_\alpha(\bx)\wrho_\beta(\by)\dbx\dby.
\end{split}
\]
We bound each of the three integrated terms on the right. Using \eqref{eq:zeta}, the first admits the lower-bound in terms of the weighted Laplacian -- weighted by the vector of masses $\scrM=\{M_\al\}_{\al\in {\mathcal I}}$,
\begin{align*}
\sum_{\alpha\neq \beta}\aphi_{\alpha\beta}\iint &|u_\alpha(\bx)-\overline{u}_\alpha|^2 \wrho_\alpha(\bx)\wrho_\beta(\by)\dbx\dby
 = \sum_\alpha \left(\sum_{\beta\neq \alpha} \aphi_{\alpha\beta}M_\beta\right) \int |u_\alpha(\bx)-\overline{u}_\alpha|^2 \wrho_\alpha(\bx)\dbx \\
& = \sum_\alpha \dzA{\al} \int |u_\alpha(\bx)-\overline{u}_\alpha|^2 \wrho_\alpha(\bx)\dbx \\
& \geq \lambda_2(\DelMA)\frac{\zetaM}{M}\sum_{\alpha,\beta}  \iint |u_\alpha(\bx)-\overline{u}_\alpha|^2 \wrho_\alpha(\bx)\wrho_\beta(\by)\dbx\dby.
\end{align*}
Similarly, the third integrand is lower-bounded by
\begin{align*}
\sum_{\alpha\neq \beta}\aphi_{\alpha\beta}\iint |u_\beta(\bx)-\overline{u}_\beta|^2 &\wrho_\alpha(\bx)\wrho_\beta(\by)\dbx\dby
 =  \sum_\beta \dzA{\bet} \int |u_\beta(\bx)-\overline{u}_\beta|^2 \wrho_\beta(\bx)\dbx \\
& \geq \lambda_2(\DelMA)\frac{\zetaM}{M}\sum_{\alpha,\beta}  \iint |u_\beta(\bx)-\overline{u}_\beta|^2 \wrho_\alpha(\bx)\wrho_\beta(\bx)\dbx\dby.
\end{align*}
Finally, by the scalar weighted Poincar\'{e} inequality \eqref{eq:Poin}, we bound the second integrand
\begin{align*}
\sum_{\alpha\neq \beta} \aphi_{\alpha\beta}\iint |\overline{u}_\alpha-\overline{u}_\beta|^2 \wrho_\alpha(\bx)\wrho_\beta(\by)\dbx\dby
& = \sum_{\alpha\neq\beta}\aphi_{\alpha\beta}|\overline{u}_\alpha-\overline{u}_\beta|^2 M_\alpha M_\beta  \\
 & \geq \frac{\lambda_2(\DelMA)}{M} \sum_{\alpha,\beta}\iint|\overline{u}_\alpha-\overline{u}_\beta|^2 \wrho_\alpha(\bx)\wrho_\beta(\by)\dbx\dby.
\end{align*}
Adding the last three lower-bounds we end up with 
\begin{align*}
\sum_{\alpha\neq \beta} & \aphi_{\alpha\beta}\iint |u_\alpha(\bx)-u_\beta(\by)|^2\wrho_\alpha(\bx)\wrho_\beta(\by)\dbx\dby\\
& \geq \lambda_2(\DelMA)\frac{\zetaM}{M}\sum_{\alpha,\beta}\iint \Big(|u_\alpha(\bx)-\overline{u}_\alpha|^2 + |\overline{u}_\alpha-\overline{u}_\beta|^2 + |\overline{u}_\beta-u_\beta(\by)|^2\Big)\wrho_\alpha(\bx)\wrho_\beta(\by)\dbx\dby\\
& = \lambda_2(\DelMA)\frac{\zetaM}{M}\iint |u_\alpha(\bx)-u_\beta(\by)|^2\wrho_\alpha(\bx)\wrho_\beta(\by)\dbx\dby,
\end{align*}
thus proving \eqref{eq:alphaneqbeta}.
\end{proof}

\begin{rmk}[{\bf Alignment and de-alignment}]\label{rem:dealign} The weighted Poincar\'{e} inequality \eqref{eq:alphaneqbeta} involves the threshold  $\displaystyle \rate_{\APhi}=\lambda_2(\DelMA)\frac{\zetaM}{M}$ which is independent of $\{\aphi_{\alpha\alpha}\}$: if $\APhi$ is connected then the non-diagonal fluctuation terms dominate the self-interacting fluctuations. In fact, this means that we can add self-fluctuations with  \emph{negative} amplitudes:\newline
 assume that $\displaystyle \left\{\begin{array}{ll}\aphi_{\al\bet} \geq 0, & \al\neq \bet,\\
\aphi_{\alpha\bet} \geq - \frac{1}{2}\rate_{\APhi}, & \al=\bet,\end{array}\right.$ then \eqref{eq:alphaneqbeta} still survives
\[
\begin{split}
\sum_{\alpha,\beta}\aphi_{\alpha\beta} \iint &|u_\alpha(\bx)-u_\beta(\by)|^2\wrho_\alpha(\bx)\wrho_\beta(\by)\dbx\dby \\
& \geq \frac{1}{2}\rate_{\APhi}\sum_{\alpha,\beta} \iint |u_\alpha(\bx)-u_\beta(\by)|^2\wrho_\alpha(\bx)\wrho_\beta(\by)\dbx\dby, \qquad \rate_{\APhi}= \lambda_2(\DelMA) \frac{\zetaM}{M}.
\end{split}
\]
\end{rmk}

%%%%%%%%%%%%%%%%%%%%%%%%%%%
\section{Smooth solutions must flock}\label{sec:hydro_flocking}
%%%%%%%%%%%%%%%%%%%%%%%%%%%%%

In this section, we prove the main flocking statement in theorem \ref{THM_1}.
The key observation is that the decay of  both -- the energy and uniform fluctuations  are dictated by the  connectivity of the multi-species configuration.
To this end, let $\St$ denote the spatial diameter of the multi-species crowd at time $t$
\begin{equation}\label{eq:diam}
D(t):= \max_{\bx,\by\in {\mathcal S}(t)}|\bx-\by|, \qquad {\mathcal S}(t)= \cup_\al \text{supp}\{\rho_\al(t,\cdot)\}.
\end{equation}
Then  $\Phi(\St)=\{\phi_{\al\bet}(\St)\}$
quantifies the minimal amplitude of communication between species $\al$ and $\bet$ at time $t$.
Our first result quantifies a minimal amount of  connectivity which implies  the decay of energy fluctuations 
\begin{equation}\label{eq:delE}
\delE(t):=  \sum_{\al,\bet\in{\mathcal I}}\iint|\bu_\al(t,\bx)-\bu_\bet(t,\by)|^2\rho_\al(t,\bx)\rho_\bet(t,\by)\dbx\dby.
\end{equation}

\begin{thm}[{\bf Decay of energy fluctuations}]\label{THM_4_1}\mbox{ }\newline
Let $(\rho_{{\al}}(t,\cdot),\bu_{{\al}}(t,\cdot))\in L_+^{1}(\rr^d)\times W^{1,\infty}(\rr^d), \al \in \mathcal{I}$, be a strong solution of the multi-species dynamics  \eqref{Hydrodynamic_Flocking_eqs}, subject to initial conditions $(\rho_{\al 0},\bu_{\al 0})$ with initial energy fluctuations $\delEz=\delE(0)$.
Then we have the apriori bound
\begin{equation}\label{u_infty}
\delE(t) \leq \delEz \cdot exp\, \Big\{\displaystyle -2\zetaM\int^t_0 \lambda_2(\DelMP(D(\tau)))\d \tau\Big\}, \qquad \zetaM=1-\frac{\max_\al M_\al}{\sum_\al M_\al}.
\end{equation}
In particular, if the crowd dynamics satisfies a `fat-tail' connectivity condition of Pareto type (but observe the dependence on $D(r)$ in contrast to \eqref{eq:Pareto})
\begin{equation}\label{eq:pareto}
 \lambda_2(\DelMP(D(r))) \gtrsim \frac{1}{(1+r)^{\theta}}, \qquad  \theta<1, 
\end{equation}
then  $\delE(t)$  decays at fractional-exponential rate
\begin{equation}\label{eq:Efrac}
\delta E(t) \lesssim \delEz\cdot e^{\displaystyle -2\rate_1\!\cdot\! t^{1-\theta}}, \qquad \rate_1=\frac{\zetaM}{1-\theta}.
\end{equation}
\end{thm}
\begin{rmk}
Again, we observe that while  the diagonal terms in $\delE$ on the left of \eqref{u_infty} account  for fluctuations within the same species, $\iint \sum_{\al=\bet} |\bu_\al(\bx,t)-\bu_\bet(\by,t)|^2\rho_\al\rho_\bet \dbx\dby$, the upper-bound on the right of \eqref{u_infty} involves $\lambda_2(\DelMP)$ which is independent of (the amplitude of) the self-interaction terms, $\{\phi_{\al\al}\}$. One learns about the behavior of its own species by its reflection through interactions with the other connected species. In fact, arguing in view of remark \ref{rem:dealign} we can even allow for self-interactions with \emph{de-alignment}, $\displaystyle \phi_{\al\al} \geq -\lambda_2(\DelMP)\frac{\zetaM}{2M}$, and yet the overall inter-species alignment will override, yielding that the crowd will align towards $\buinf$.  
\end{rmk}
\begin{proof} Since the total mass, $\displaystyle \Min=\sum_\alpha \int\rho_\alpha(t,\bx) \dbx$, and total momentum, $\displaystyle \sum_\alpha \int \rho_\alpha(t,\bx) \bu_\alpha(t,\bx) \dbx$, are conserved in time, it follows that the decay rate of the fluctuations is the same as the decay rate of the total kinetic energy,
\bel\label{eq:energy}
\frac{\D}{\D t} \delta E(t) = 2\Min \frac{\D}{\D t}E(t), \qquad E(t):=\sum_{\alpha\in \mathcal{I}} \int \rho_\alpha(t,\bx) |\bu_\alpha(t,\bx)|^2 \dbx.
\eel
  A straightforward computation using the multi-species dynamics  \eqref{Hydrodynamic_Flocking_eqs} yields  
\begin{align*}
\frac{\D}{\D t}\bigg(&\sum_{\al\in\mathcal{I}}\int \rho_\al |\bu_\al|^2 \dbx\bigg)=
2\int \sum_{\al,\beta\in\mathcal{I}}\big\langle \rho_\al \bu_\al,\, \phi_{\alpha\beta}*(\rho_\beta \bu_\beta)- (\phi_{\alpha\beta}*\rho_\beta) \bu_\al \big\rangle\dbx\nonumber\\
=& 2\iint \sum_{\al,\beta\in\mathcal{I}}\Big(\big\langle \rho_\al(\bx) \bu_\al(\bx),\, \phi_{\alpha\beta}(|\bx-\by|)\rho_\beta(\by)\bu_\beta(\by)\big\rangle \\
 & \hspace*{7.9cm} -\rho_\al(\bx)|\bu_\al(\bx)|^2 \phi_{\alpha\beta}(|\bx-\by|)\rho_\beta(\by)\Big)\dbx\dby\nonumber\\
=& 2\iint \sum_{\al,\beta\in\mathcal{I}}\big\langle \rho_\al(\bx) \bu_\al(\bx),\,\phi_{\alpha\beta}(|\bx-\by|)\rho_\beta(\by) \bu_\beta(\by)\big\rangle\dbx\dby\nonumber \\
 & -\iint \sum_{\al,\beta\in\mathcal{I}}\Big(\rho_\al(\bx)|\bu_\al(\bx)|^2 \phi_{\alpha\beta}(|\bx-\by|)\rho_\beta(\by) +  \rho_\beta(\by)|\bu_\beta(\by)|^2 \phi_{\beta\alpha}(|\bx-\by|)\rho_\al(\bx)\Big) \dbx\dby\nonumber\\
=& -\iint \sum_{\al,\beta\in\mathcal{I}}\phi_{\alpha\beta}(|\bx-\by|) |\bu_\al(t,\bx)-\bu_\beta(t,\by)|^2\rho_\al(t,\bx)\rho_\beta(t,\by)\dbx\dby. \nonumber
\end{align*}
Since $\phi_{\al\bet}$ are decreasing, 
$\phi_{\alpha\beta}(|\bx-\by|) \geq \phi_{\al\bet}(\St)$, hence
\begin{equation}\label{eq:thisisE}
\frac{\D}{\D t}E(t) \leq  - \sum_{\al,\beta\in\mathcal{I}}\phi_{\alpha\beta}(D(t)) \iint|\bu_\al(t,\bx)-\bu_\beta(t,\by)|^2\rho_\al(t,\bx)\rho_\beta(t,\by)\dbx\dby.
\end{equation}
We now appeal to the vector-function version of Poincar\'{e} inequality in Lemma \ref{lem:Poinv}, obtaining\footnote{To be precise, here one employs the \emph{vector} statement
\[
\frac{\lambda_2(\DelMP)}{\sum_\al M_\al}=\min_{|\delbu|_M=1} \sum_{\al \neq \bet \in {\mathcal I}}\Phi_{\al\bet}|{\mathbf u}_\al-{\mathbf u}_\bet|^2M_{\al}M_\bet,   
\qquad |\delbu|_M^2= \sum_{\al\neq \bet \in {\mathcal I}}|{\mathbf u}_\al-{\mathbf u}_\bet|^2M_{\al}M_\bet, \quad {\mathbf u} \in \rr^d,
\]
which follows by aggregating the scalar components of \eqref{eq:alphaneqbeta} (as was done in \cite[Sec 3.1]{CS07}).}
\[
\frac{1}{2M}\frac{\D}{\D t}\delE(t) \leq  -\lambda_2(\DelMP(D(t)))\frac{\zetaM}{M} \delE(t),
\]
and the desired bound \eqref{u_infty} follows.
\end{proof}

\noindent
The decay of energy fluctuations, $\delE(t)$, implies  decay of pointwise fluctuations
\[
\delV(\bu(t))=\max_{\al,\beta\in\mathcal{I}}\max_{\bx,\by\in{\mathcal S}(t)}|\bu_\al(t,\bx)-\bu_\beta(t,\by)|.
\]
\begin{thm}[{\bf Decay of uniform fluctuations}]\label{THM_4_2}\mbox{ }\newline
Let $(\rho_{{\al}}(t,\cdot),\bu_{{\al}}(t,\cdot))\in L_+^{1}(\rr^d)\times W^{1,\infty}(\rr^d), \al \in \mathcal{I}$, be a strong solution of the multi-species dynamics  \eqref{Hydrodynamic_Flocking_eqs}, subject to initial conditions $(\rho_{\al 0},\bu_{\al 0})$, and assume the crowd dynamics satisfies the `fat-tail' connectivity condition \eqref{eq:pareto}. Then  $\delV(\bu(t))$ decays at fractional-exponential rate: there exist constants $\displaystyle C_2=C(\max_{\al,\bet}\phi_{\al\bet}(0), M)>0$ and  $\rate_2=\rate(\theta,M)>0$ such that
\begin{equation}\label{eq:vfrac}
\delV(\bu(t)) \lesssim C_2\cdot \delV_0\cdot e^{\displaystyle -2\rate_2\cdot t^{1-\theta}}, \qquad \delVz=\delV(\bu(0)). 
\end{equation}
\end{thm}

\begin{proof}
We consider the strong solution $(\rho_\al,\bu_\al)$ in the non-vacuous region
$\bx,\by\in {\mathcal S}$, where the alignment terms on the right of \eqref{Hydrodynamic_Flocking_eqs} admits the usual commutator form \cite{ShvydkoyTadmorI}
\bel\label{EQ:hydrodynamic_flocking}
\pa_t  \bu_\al+ (\bu_\al\cdot \nabla) \bu_\al=\sum_{\beta\in \mathcal{I}}\{
\phi_{\alpha\beta}*(\rho_\beta \bu_\beta)-(\phi_{\alpha\beta}*\rho_\beta)\bu_\al\}, \qquad \forall\alpha,\beta\in {\mathcal I}.
\eel
Arguing along the lines of  \cite{HeTadmor17}, we first fix an arbitrary unit vector $\mathbf{w}\in \mathbb{R}^d$ and project \eqref{EQ:hydrodynamic_flocking} onto the space spanned by $\mathbf{w}$ to get
\begin{align*}
(\pa_t+\bu_\al\cdot\na)\lan \bu_\al(t,\mathbf{x}),\mathbf{w}\ran=\sum_{\beta\in\mathcal{I}}\int\phi_{\al\beta}(|\bx-\by|)(\lan\bu_\beta(t,\by),\mathbf{w}\ran-\lan\bu_\al(t,\bx),\mathbf{w}\ran)\rho_\beta(t,\by)\mathrm{d}\by.
\end{align*} 
Now we assume that $\lan\bu_{\al}(t,\bx),\mathbf{w}\ran$ reaches a maximum  value at $(\bx(t),\al(t))=(\bx_+(t),\al_+(t))$ and a minimum value at $((\bx(t),\al(t))=(\bx_-(t),\al_-(t)))$, denoting
\[
u_+(t):=\max_{\al\in {\mathcal I}}\sup_{\bx\in {\mathcal S}(t)}\lan\bu_\al(t,\bx), \mathbf{w}\ran=\bu_{\al_+(t)}(\bx_+(t)).
\] 
We abbreviate $c_{\al\bet}(t):=\phi_{\al\bet}(D(t))$ and  $\displaystyle \overline{\bu}_\bet(t):=\frac{1}{M_\bet}\int \rho_\bet\bu_\bet(t,\by)\dby$. Direct computation  of the time evolution of $u_+(t)$ yields, 
\begin{equation}\label{eq:uplus}
\begin{split}
\frac{\D}{\D t}u_+(t)=&\sum_{\beta\in\mathcal{I}}\int\phi_{\al_+\beta}(|\bx_+-\by|)\big(\lan\bu_\beta(t,\by), \mathbf{w}\ran-\lan\bu_{\al_+}(t,\bx_+),\mathbf{w}\ran\big)\rho_\beta(t,\by)\mathrm{d}\by\\
\leq &\sum_{\beta\in\mathcal{I}}c_{\al_+\beta}\int \big(\lan \bu_\beta(t,\by),\mathbf{w}\ran-\lan\bu_+(t),\mathbf{w}\ran\big)\rho_\beta(t,\by)\mathrm{d}\by
\\
=&\sum_{\beta\in\mathcal{I}}c_{\al_+\beta}M_\beta \lan\overline{\bu}_\beta(t) -\bu_+(t),\mathbf{w}\ran\\
=&\sum_{\beta\in\mathcal{I}}c_{\al_+\beta}M_\beta \lan\overline{\bu}_\beta(t) -\buinf,\mathbf{w}\ran +\sum_{\beta\in\mathcal{I}}c_{\al_+\beta}M_\beta \lan\buinf-\bu_+(t),\mathbf{w}\ran=:I + II
\end{split}
\end{equation}
We proceed to show that the first term is bounded by the (rapidly decaying) energy fluctuations while the  second term will contribute to the pointwise fluctuations. Indeed, since 
\[
c_{\al\bet}(t) \leq \max_{\al,\bet}\phi_{\al\bet}(D_0)=:C_{\phi},
\]
 and 
$\displaystyle M_\beta \big(\overline{\bu}_\beta(t) -\buinf\big)\equiv \frac{1}{M}\sum_\al \iint (\bu_\bet(t,\by)-\bu_\al(t,\bx))\rho_\al(t,\bx)\rho_\bet(t,\by) \dbx\dby$, then by Cauchy-Schwarz  we find
\begin{align*}
I \leq\, &\frac{C_{\phi}}{M}\,\sum_{\al,\bet} \Big(\iint|\bu_\bet(t,\by)-\bu_\al(t,\bx)|^2\rho_\al(t,\bx)\rho_\bet(t,\by) \dbx\dby\Big)^{1/2}\Big(\iint \rho_\al(t,\bx)\rho_\bet(t,\by) \dbx\dby\Big)^{1/2}\\
\leq\, & \frac{C_{\phi}}{M}\,\Big(\sum_{\al,\bet} \iint|\bu_\bet(t,\by)-\bu_\al(t,\bx)|^2\rho_\al(t,\bx)\rho_\bet(t,\by) \dbx\dby \times \sum_{\al,\bet}\iint \rho_\al(t,\bx)\rho_\bet(t,\by) \dbx\dby\Big)^{1/2}\\
& =C_{\phi} \big(\delE(t)\big)^{1/2}.
\end{align*}
On the other hand, since $\lan \buinf-\bu_+,\mathbf{w}\ran\leq0$, we use  the reversed lower bound \eqref{eq:zeta} 
\begin{align*}
II \leq \dzPD{\al_+}\lan\buinf -\bu_+(t),\mathbf{w}\ran
\leq \zetaM\lambda_2(\DelMP(D(t))) \big(\overline{u}_\infty-u_+(t)\big), \quad \overline{u}_\infty:=\lan\buinf,{\mathbf w}\ran.
\end{align*}
The last two inequalities yield
\[
\frac{\D}{\D t}u_+(t) \leq C_{\phi}\big(\delE(t)\big)^{1/2} +\zetaM\lambda_2(\DelMP(D(t))) \big(\overline{u}_\infty-u_+(t)\big);
\]
 similarly, we estimate the time evolution of 
$\displaystyle u_-(t):=\min_{\al\in{\mathcal I}}\inf_{\bx\in{\mathcal S}}\lan\bu_\al(t,\bx), \mathbf{w}\ran$ obtaining
\[
\frac{\D}{\D t}u_-(t) \geq -C_{\phi}\big(\delE(t)\big)^{1/2} +\zetaM\lambda_2(\DelMP(D(t))) \big(\overline{u}_\infty-u_-(t)\big).
\]
The difference of the last two bounds yields the apriori bound on $\delV(u(t)):=u_+(t)-u_-(t)$, 
\begin{equation}\label{eq:aprioridelV}
\frac{\D}{\D t}\delV(u(t))\leq -\zetaM\lambda_2(\DelMP(D(t)))\cdot \delV(u(t))+
2C_{\phi} (\delE(t))^{1/2}.
\end{equation}
Observe that  $\displaystyle \delV(u(t))=\max_{\al,\bet\in{\mathcal I}}\sup_{\bx,\by\in {\mathcal S}(t)}\langle {\bu}_\al(t,\bx)- {\bu}_\bet(t,\by),\bw\rangle$ is   the diameter of projected velocities on arbitrary unit vector  $\bw$. The assumed \eqref{eq:pareto} implies that $\delE(t)$ admits 
 the fractional exponential decay  \eqref{eq:Efrac}, and we end up with,
\begin{equation}\label{eq:dvde}
\ddt\delV(\bu(t)) \leq -\zetaM\lambda_2(\DelMP(D(t)))\cdot\delV(\bu(t)) + 2C_{\phi}\cdot
(\delEz)^{1/2}e^{\displaystyle -\rate_1\!\cdot\! t^{1-\theta}}.
\end{equation}
Finally, $(\delE_0)^{1/2}\leq M\cdot \delVz$ and by assumption $\lambda_2(\DelMP(D(t))) \gtrsim(1+t)^{-\theta}$, hence \eqref{eq:vfrac} follows by integration of \eqref{eq:dvde}.
\end{proof}
\begin{rmk} Revisiting \eqref{eq:uplus} we find 
\begin{align*}
\frac{\D}{\D t}u_+(t)\leq&\sum_{\beta\in\mathcal{I}}\phi_{\al_+\beta}(D(t))M_\beta \lan\overline{\bu}_\beta(t) -\bu_+(t),\mathbf{w}\ran \leq \dzPD{\al_+}\max_{\beta\in\mathcal{I}}\lan\overline{\bu}_\beta(t) -\bu_+(t),\bw\ran\\
\leq &\zetaM\lambda_2(\DelMP(D(t)))\max_{\beta\in\mathcal{I}}\lan\overline{\bu}_\beta (t)-\bu_+(t),\bw\ran,
\end{align*}
and likewise
\begin{align*}
\frac{\D}{\D t}u_-(t)
\geq  \zetaM\lambda_2(\DelMP(D(t)))\min_{\beta\in\mathcal{I}}\lan\overline{\bu}_\beta(t)-\bu_-(t),\bw\ran.
\end{align*}
The difference of the last two estimates yield the apriori bound
\begin{equation}\label{eq:flucE}
\frac{\D}{\D t}\delV(u(t)\leq \zetaM\lambda_2(\DelMP(D(t)))\cdot\Big(-\delV(u(t))+
\delV(\overline{u}(t))\Big), \qquad \delV(\overline{u}):=\overline{u}_+-\overline{u}_-.
\end{equation}
Since the diameter of averaged velocities $\delV(\overline{u})$ is smaller than the diameter of the velocities $\delV(u)$, \eqref{eq:flucE} implies  that the pointwise velocity diameter does not increase
\begin{equation}\label{eq:pointwise}
 \delV(\overline{\bu}(t)) \leq \delV(\bu(t)) \ \leadsto \ \delV(\bu(t)) \leq \delV_0. 
\end{equation}
Note that the apriori bound \eqref{eq:pointwise} does not require any connectivity assumption;  theorem \ref{THM_4_2}  quantifies how an additional  `fat-tail' connectivity \eqref{eq:pareto} enforces the fractional exponential decay of $\delV(\bu(t))$.
\end{rmk}

The last two theorems still require information on the dynamic growth of the supports ${\mathcal S}(t)=\cup_\al\text{supp}\, \{\rho_\al(t,\cdot)\}$, in order to access the possible growth  of $\St$ and the corresponding decay  of 
$\phi_{\al\bet}(\St)$ in \eqref{eq:pareto}. Our next result provides apriori bound how on far the different species can spread out, and this enables us to quantify flocking in terms of the connectivity of $\{\phi_{\al\bet}(r)\}$, \emph{independent} of the diameter dynamics. To this end, observe that according to the apriori bound \eqref{eq:pointwise},
 the  velocities of the different species remain bounded, 
  and hence the spatial diameter of  the  support  of the crowd can grow at most linearly in time: indeed, tracing the particle paths $(\bx(t),\by(t))\in {\mathcal S}$ yields 
\begin{equation}\label{eq:delVt}
\ddt \St \lesssim \delV(\bu(t)) \ \ \leadsto \ \ \St= \max_{\bx,\by\in {\mathcal S}(t)} |\bx-\by| \lesssim D_0 + \delV_0\cdot t.
\end{equation}
We conclude  the lower-bound (recall that $\phi_{\al\bet}$ are decreasing)
$\phi_{\al\bet}(\St)  \gtrsim \phi_{\al\bet}\big(D_0+\delV_0 \cdot t\big)$. 
We are now ready to prove theorem \ref{THM_1}.
\begin{proof}[Proof of theorem \ref{THM_1}\nopunct] proceeds in three steps.

\medskip\noindent
{\bf Step \#1. Fractional exponential decay}. The variational   characterization of the  Fiedler number  \eqref{eq:CF}, implies that
$\lambda_2(\cdot)$ is an increasing function of the non-negative entries in its argument,  
\begin{equation}\label{eq:delCandV}
\begin{split}
\frac{\lambda_2(\DelMP(D(t)))}{\Min} & = \min_{|\delbu|_M=1} \sum_{\al,\bet}\phi_{\al\bet}(\St)\cdot |u_\al-u_\bet|^2M_{\al}M_\bet
\\
 & \gtrsim  \min_{|\delbu|_M=1} \sum_{\al,\bet}\phi_{\al\bet}\big(D_0+\delV_0\cdot t\big)\cdot |u_\al-u_\bet|^2M_{\al}M_\bet \qquad \\
 & =
 \frac{\lambda_2\big(\DelMP\big(D_0+\delV_0 \cdot t\big)\big)}{\Min}.
\end{split}
\end{equation}
Hence,  the  Pareto decay  $\lambda_2(\DelMP(r))\gtrsim (1+r)^{-\theta}$ assumed in \eqref{eq:Pareto} implies $\lambda_2(\DelMP(D(t))) \gtrsim (1+D_0+\delV_0\cdot t)^{-\theta}$ and  the apriori estimate \eqref{u_infty} implies
\[
\delta E(t) \lesssim \delEz\cdot e^{\displaystyle -2\rate_3\!\cdot\! t^{1-\theta}} , \quad \rate_3:= \frac{\zetaM}{(1-\theta)\cdot\delV_0}. 
\]

\noindent
{\bf Step \#2. Finite diameter}. The Pareto-type condition \eqref{eq:Pareto} implies an improved flocking rate of full exponential rate. Indeed,  the apriori bound \eqref{eq:aprioridelV} together with \eqref{eq:delCandV} yield
\[
\ddt\delV(\bu(t)) \lesssim -(1+D_0+\delV_0\cdot t)^{-\theta} \cdot\delV(\bu(t)) +  2C_\phi\cdot (\delta E_0)^{1/2}e^{\displaystyle -\rate_3\cdot t^{1-\theta}}.
\]
As before we use $(\delta E_0)^{1/2} \leq M\cdot \delV_0$; integrating the last inequality we find that $\delV(\bu(t))$ satisfies a fractional exponential decay 
\[
\delV(\bu(t)) \lesssim \delV_0\cdot e^{\displaystyle -\rate_4\!\cdot\! t^{1-\theta}}, \qquad \rate_4=\min\{\rate_1,\rate_3\}>0
\]
 which in turn implies  a bounded spatial diameter uniformly in  time\footnote{Tracing the dependence of $C_\theta$ on $\theta$ we find
 $\displaystyle C_\theta \lesssim \int_0^\infty e^{-\rate_4\cdot t^{1-\theta}}\dt$ with $\rate_4 \lesssim \frac{1}{1-\theta}$ which yield\newline $C_\theta \sim (1-\theta)^{\frac{\theta}{1-\theta}}$.},
\begin{equation}\label{eq:Dbd}
\ddt \St \leq \delV(\bu(t)) \lesssim \delV_0\cdot e^{\displaystyle -\rate_4\!\cdot\! t^{1-\theta}}  \ \leadsto  \ \St \leq D_\infty \, \leq  D_0+C_\theta \cdot\delV_0 < \infty.
\end{equation}
\noindent
{\bf Step \#3. Exponential decay}. We now have a \emph{uniform} lower bound on the minimal communication, $\phi_{\al\bet}(D(t)) \geq \phi_{\al\bet}(D_\infty)$. Hence, the monotone increasing dependence of $\lambda_2(\DelMA)$ on the entries of $\APhi$, consult \eqref{eq:CF}, implies
\begin{equation}\label{eq:laminf}
\lambda_2(\DelMP(D(t))) \geq \lambda_2(\DelMP_\infty)>0, \qquad
\Phi_\infty:=\{\phi_{\al\bet}(D_\infty)\}.
\end{equation}
We revisit the energy apriori fluctuations bound \eqref{u_infty}, obtaining the exponential decay
\[
\delta E(t) \leq \delEz \cdot e^{\displaystyle -2\rate t}, \qquad \rate=\zetaM\lambda_2(\DelMP_\infty).
\]
Since $\displaystyle 
\sum_\al\int|\bu_\al(t,\bx)-\overline{\bu}_\infty|^2\rho_\al(t,\bx)\dbx
\equiv\frac{1}{2\Min}\delta E(t)$, exponential flocking \eqref{eq:revisit} follows. Moreover, revisiting the uniform fluctuations \eqref{eq:vfrac} with \eqref{eq:laminf} yields the exponential decay
\begin{equation}\label{eq:vexp}
\max_{\al\in{\mathcal I}}\sup_{\bx\in{\mathcal S}(t)}|\bu_\al(t,\bx)-\overline{\bu}_\infty| \lesssim \delV_0\cdot e^{\displaystyle -\rate t}.
\end{equation}
\end{proof}

%%%%%%%%%%%%%%%%%%%%%%%%%%%%%%
\section{Existence of global smooth solutions}\label{sec:global}
%%%%%%%%%%%%%%%%%%%%%%%%%%%%
\subsection{Critical threshold in one-dimensional flocking dynamics}
\begin{proof}[Proof of Theorem \ref{Thm_2}]
Taking  spatial derivative of the momentum equation \eqref{Hydrodynamic_Flocking_eqs} yields
\begin{align}
(\pa_t+u_\al\pa_x)\left(\pa_x u_\al+\sum_{\beta\in\mathcal{I}} \phi_{\al\beta}*\rho_\beta\right)=-\pa_x u_\al\left(\sum_{\beta\in\mathcal{I}}\phi_{\al\beta}*\rho_\beta+\pa_x u_\al\right),\quad \forall \al\in \mathcal{I}.\label{EQ:na_x_u_al}
\end{align}
Thus, the ``$e$''-quantities, $e_\al:=\pa_x u_\al+\sum_{\beta} \phi_{\al\beta}*\rho_\beta$ satisfy $\pa_t e_\al +\pa_x(u_\al e_\al)=0$ and pairing it with the mass equations $\pa_t \rho_\al +\pa_x(u_\al \rho_\al)=0$ yields 
\[
\pa_t q_\al+u_\al \pa_xq_\al=0, \qquad q_\al:=\frac{e_\al}{\rho_\al}.
\]
It follows that $q_\al\geq 0$ and hence $e_\al\geq0$ are  invariant zones: if $e_\al(t=0,x) \geq 0$ for all $x\in \rt$  then
\begin{align}
\pa_x u_{\al}+\sum_{\beta\in\mathcal{I}} \phi_{\al\beta}*\rho_\beta \geq 0,\quad \forall t\geq 0.
\end{align}
Moreover, arguing along the lines of \cite[sec. 3]{ShvydkoyTadmorI} 
\begin{align*}
\pa_t\rho_\al+u_\al\pa_x\rho_\al &= - \pa_x u_\al\rho_\al
= -\left(e_\al-\sum_{\beta\in\mathcal{I}} \phi_{\al\beta}*\rho_\beta\right)\rho_\al
 = -q_\al\rho^2_\al +\rho_\al \sum_{\beta\in\mathcal{I}} \phi_{\al\beta}*\rho_\beta, 
\end{align*}
and the uniform bound $|e_\al/\rho_\al(t,\cdot)|_\infty \leq |e_\al/\rho_\al(0,\cdot)|_\infty <\infty$ reveals that $\rho_\al$ remains bounded away from vacuum.

Since $\phi_{\al\bet}$ are uniformly bounded, we obtain the  lower bound,
\begin{align}
\pa_x u_\al(x,t)\geq -\sum_{\beta\in\mathcal{I}}|\phi_{\al\beta}|_\infty M_\beta,\quad \forall (t,x)\in (\rr_+,\rt), \al\in \mathcal{I}.
\end{align}
On the other hand we can see directly from the equation \eqref{EQ:na_x_u_al} that $\pa_x u_\al$ has an upper bound for all time. Combining this with the lower bound, we have that $|\pa_x u_{\al}|_\infty\leq C<\infty$ for all time and  the existence of strong solutions follows.
\end{proof}

\subsection{Critical threshold in two-dimensional flocking dynamics}
\begin{proof}[Proof of Theorem \ref{Thm_4}]
Our purpose is to show that the derivatives $\{\pa_j\bu_{\al}^i\}$ are uniformly bounded. We proceed in four steps along the lines of \cite{HeTadmor17} for the case of two-dimensional single species dynamics. 

\medskip\noindent
\underline{Step \#1} --- the dynamics of $\diver{\bu_\al}+\sum_{\beta\in\mathcal{I}}\phi_{\al\beta}*\rho_\beta$. Differentiation of \eqref{EQ:hydrodynamic_flocking} implies that the velocity gradient matrix, $(\na \mathbf{u}_\al)_{ij}=\pa_j \bu_{\al}^i$, satisfies

\bel\label{eq:CSM}
(\na \bu_\al)_t+\bu_\al\cdot\na (\na \bu_\al)+(\na \bu_{\al})^2=-\sum_{\beta\in \mathcal
{I}} \phi_{\al\beta}*\rho_\beta \na \bu_\al+R_\al,
\eel
where the entries of the residual matrices 
\[
 (R_{\al})_{ij}:=\sum_{\beta\in \mathcal{I}}\int {\pa_j\phi_{\al\beta}(|\bx-\by|)}(\bu_{\beta}^i(\by)-\bu^i_\al(\bx))\rho_\beta(\by)\d\by,\nonumber
\]
do not exceed $|(R_{\al})_{ij}|\leq \sum_{\beta\in\mathcal{I}}|\phi'_{\al\beta}|_\infty M_\beta \cdot \delV(t)$.
The entries of the residual matrix $\{(R_{\al})_{ij}\}$ can be estimated using the exponentially decaying velocity fluctuations \eqref{eq:vexp} 
\begin{align}\label{R_bound}
|(R_{\al})_{ij}|\leq \sum_{\beta\in\mathcal{I}}|\phi'_{\al\beta}|_\infty M_\beta \cdot \delV(t) \lesssim \delV_0\cdot e^{-2\rate t}.
\end{align}

The first step is to bound the divergence: taking the trace of \eqref{eq:CSM} we find that   $\d_\al:=\nabla\cdot \bu_\al$ satisfies
\begin{align}\label{eq:CSd}
(\pa_t+ \u_\al\cdot\nabla) \d_\al + \trace{(\na\bu_\al)^2} = - \left(\sum_{\beta\in\mathcal{I}}\phi_{\al\beta}*\rho_\beta\right)\d_\al +\trace{R_\al}.
\end{align}
Arguing along the lines of  \cite{CCTT16} we invoke the mass equation and obtain the following relation, 
\begin{align*}
\trace{R}_\al= &\sum_{\beta\in\mathcal{I}}\phi_{\al\beta}* \nabla\cdot (\rho_\beta \u_\beta) - \sum_{\beta\in\mathcal{I}}\u_\al\cdot\nabla \phi_{\beta}*\rho_\beta = 
-\left(\sum_{\beta\in\mathcal{I}}\phi_{\al\beta}* \rho_\beta\right)_t -\u_\al\cdot\nabla \left(\sum_{\beta\in\mathcal{I}}\phi_{\al\beta}*\rho_\beta \right)\\
=& -\left(\sum_{\beta\in\mathcal{I}}\phi_{\al\beta}*\rho_\beta\right)',
\end{align*}
where $(\cdot)'$ denotes the material derivative, $(\cdot)':=(\pa_t+\bu_\al\cdot\na_\bx)(\cdot)$. 
Similar to \cite{HeTadmor17}, we define the following two quantities
\begin{equation}\label{eq:etasym}
\na \bu_\al=S_\al+\Omega_\al, \quad S_\al=\frac{1}{2}(\na \bu_\al +\na \bu_\al^\top), \quad \Omega_\al:=\left(\begin{array}{cc}0 & -\omega_\al\\ \omega_\al & 0 \end{array}\right),
\end{equation}
where $\omega_\al$ is the scaled vorticity $\omega_\al=\frac{1}{2}(\pa_1 \bu_{\al}^2-\pa_2 \bu_\al^1)$. The symmetric part $S_\al$ has two real eigenvalue, i.e., $\lambda_1(S_\al) \leq \lambda_2(S_\al)$. Next, we recall the identity relating the trace $\trace{(\na \bu_\al)^2}$ to the \emph{spectral gap}, $\lambda_2(S_\al)-\lambda_1(S_\al)\geq 0$, \cite[eq.(2.11)]{HeTadmor17},
\begin{align}
\trace{(\na \bu_\al)^2}\equiv \frac{\d_\al^2+\etaS^2_\al-4\omega_\al^2}{2}, \qquad \etaS_\al:=\lambda_2(S_\al)-\lambda_1(S_\al)\geq 0.
\end{align} 
Expressed in terms of $\etaS_\al$, the trace dynamics \eqref{eq:CSd} now reads
\[
\left(\d_\al+\sum_{\beta\in\mathcal{I}}\phi_{\al\beta}*\rho_\beta\right)' = \hf(4\omega_\al^2 -\etaS^2_\al)-\hf \d_\al\left(\d_\al+2\sum_{\beta\in\mathcal{I}}\phi_{\al\beta}*\rho_\beta\right).
\]
This calls for the introduction  of the new ``natural'' variable $\e_\al=\d_\al+\sum_{\beta\in\mathcal{I}}\phi_{\al\beta}*\rho_\beta$, satisfying
\begin{align}\label{eq:CSe}
 \e_\al' = \hf \left(\bigg(\sum_{\beta\in\mathcal{I}}\phi_{\al\beta}*\rho_\beta\bigg)^2 +4\omega_\al^2 -\etaS^2_\al -\e_\al^2\right). 
 \end{align}
 Our purpose is to show that $\{\bx \ | \ \e_\al(t,\bx) \geq0,\enskip\forall\al\in\mathcal{I}\}$  is invariant region of the dynamics \eqref{eq:CSe}.
 
\medskip\noindent  
\underline{Step \#2} --- bounding the spectral gap $\etaS_\al$. Consider the dynamics of the symmetric part of \eqref{eq:CSM}
\[(S_\al)_t+\bu_\al\cdot \na S_\al+S_\al^2-\frac{\omega_\al^2}{4}{\mathbb I}_{2\times 2}=-\sum_{\beta\in\mathcal{I}} \phi_{\al\beta}*\rho_\beta S_\al+R_{\al, sym},\quad R_{\al, sym}=\frac{1}{2}(R_\al+R_\al^\top).
\]
The spectral dynamics of its eigenvalues $\lambda_i(S_\al)$ is governed by 
\begin{equation}\label{eq:mui}
\lambda'_i + \lambda^2_i =\omega_\al^2 -\left(\sum_{\beta\in\mathcal{I}}\phi_{\al\beta}*\rho_\beta\right)\lambda_i +\big\langle \bs_\al^i, R_{\al,\text{sym}}\bs_\al^i\big\rangle
\end{equation}
driven by  the \emph{orthonormal} eigenpair $\{\bs_\al^1,\bs_\al^2\}$ of the symmetric $S_\al$.
Taking the difference, we find that $\etaS_\al=\lambda_2(S_\al)-\lambda_1(S_\al)\geq 0$ satisfies,
\begin{equation}\label{eq:etaS}
(\etaS_\al)' + \e_\al\etaS_\al = q_\al, \qquad q_\al:= \big\langle \bs_\al^2, R_{\al,\text{sym}}\bs_\al^2\big\rangle - \big\langle \bs_\al^1, R_{\al,\text{sym}}\bs_\al^1\big\rangle.
\end{equation}
The residual term $q_\al$ is upper-bounded by the size of the entries   $\{R_{\al,j}^i\}$ in \eqref{R_bound}, $|q_\al(t,\cdot)|_\infty \leq 2 \max_{ij} |R_{\al,j}^i(t,\cdot)|_\infty \lesssim  \delV_0\cdot e^{\displaystyle -2\rate t}$.
Hence, as long as $\e_\al(t,\cdot)$ remains positive, the spectral gap does not exceed
\begin{equation}\label{eq:etaSbd}
|\etaS_\al(t,\bx)| \leq \max_\bx |\etaS_\al(0,\bx)| + Const.\frac{\delV_0}{\rate} < C_1.
\end{equation}
The first inequality on the right follows from integration of \eqref{eq:etaS}; the second follows from the assumed bound on $|\etaS_\al(0)|\leq \frac{1}{2}C_1$ in \eqref{eq:etaCT}, and our choice of small enough $\delV_0\leq C_1$, so that $\displaystyle Const.\frac{\delV_0}{\rate}\leq \frac{1}{2}C_1$; the constant $C_1$ is yet to be determined.

\medskip\noindent
\underline{Step \#3} --- The invariance of $\e_\al(t,\cdot)\geq 0$ . We return to \eqref{eq:CSe}: expressed in terms of the lower bound 
$\sum_{\beta\in\mathcal{I}}\phi_{\al\beta}*\rho_\beta \geq \sum_{\beta\in\mathcal{I}}\phi_{\al\beta}(D_\infty)M_\bet$ we find 
\begin{equation}\label{eq:eq}
\e_\al'\geq \hf \left(b_\al^2 - \e_\al^2\right), \qquad b_\al(t,\bx):=\sqrt{\left(\sum_{\beta\in\mathcal{I}}\phi_{\al\beta}(D_\infty)M_\bet\right)^2-\etaS^2_\al(t,\bx)}.
\end{equation}
Observe that $b_\al$ are well-defined: we set 
\begin{equation}\label{eq:setC0}
C_1:=\frac{1}{\sqrt{2}} \min_\al\sum_{\beta\in\mathcal{I}}\phi_{\al\beta}(D_\infty)M_\beta,
\end{equation}
so that the upper-bound \eqref{eq:etaSbd} implies 
\[
\left(\sum_{\beta\in\mathcal{I}}\phi_{\al\beta}(D_\infty)M_\bet\right)^2-\etaS^2_\al(t,\bx) \geq \frac{1}{2}C^2_0 \ \ \leadsto \ \ 
b_\al(t,\bx) \geq \cm:=\frac{1}{\sqrt{2}}C_1>0. 
\]
Since $\e_\al'  \geq \hf((\cm)^2-\e_\al^2)= \hf(\cm-\e_\al)(\cm+\e_\al)$,  it follows that $\e_\al$ is increasing whenever $\e_\al\in (-\cm,\cm)$ and in particular, if $\e_\al(0)\geq0$, $\forall\al\in\mathcal{I}$ then $\e_\al(t,\bx)$ remains positive at later times. Thus, if the initial data are \emph{sub-critical} in the sense that \eqref{eq:eCT} holds  
\[
\e_\al(0,\bx)= \diver{\bu_\al}(0,\bx)+\sum_{\beta\in\mathcal{I}}\phi_{\al\beta}*\rho_\al(0,\bx) \geq0, \quad \forall\bx\in\rr^2,
\]
then $\e_\al(t,\cdot)\geq 0$ and $\etaS_\al(t,\cdot)$ remains bounded.%7

\medskip\noindent
\underline{Step \#4} --- an upper-bound of $\e_\al(t,\cdot)$. The lower-bound  $\e_\al \geq 0$ implies that the vorticity is bounded. Indeed, the anti-symmetric part of \eqref{eq:CSM} yields that the vorticity $\omega_\al$ satisfies  
\begin{equation}\label{eq:vorticity}
\omega_\al' +\e_\al\omega_\al= \hf\trace{JR_\al}, \qquad J=\left(\begin{array}{cc} 0 & -1 \\ 1 &0\end{array}\right)
\end{equation}
hence applying \eqref{R_bound} yields
\begin{equation}\label{eq:vort}
|\omega_\al|'\leq  -\e_\al|\omega_\al| + \hf|q_\al|, \qquad |q_\al(t,\cdot)|\lesssim \delV_0\cdot e^{\displaystyle -2\rate t}
\end{equation}
and we end up with same upper-bound on $\omega_\al$ as with $\etaS_\al$,
\begin{equation}\label{eq:vorbd}
|\omega_\al(t,\cdot)|_\infty \leq (\omega_\al)_+, \qquad (\omega_\al)_+:= \max_\bx |\omega_\al(0,\bx)| +  \hf C_1.
\end{equation}
Returning to \eqref{eq:CSe} we have 
\[
\e'_\al \leq \hf \Bigg(\bigg(\sum_{\beta\in\mathcal{I}}\phi_{\al\beta}*\rho_\beta\bigg)^2 +4\omega_\al^2 -\e_\al^2\Bigg) \leq \hf \Bigg(\bigg(\sum_{\beta\in\mathcal{I}}|\phi_{\al\beta}|_\infty M_\beta\bigg)^2+ 4(\omega_\al)_+^2 -\e_\al^2\Bigg),
\]
which implies that $|\e_\al(t,\cdot)|_\infty \leq (\e_\al)_+ <\infty$.
The uniform bound on $\e_\al$ implies that $\diver{\bu_\al}$ is uniformly bounded,
$|\diver{\u_\al}| \leq |\e_\al|_\infty+\sum_{\beta\in\mathcal{I}}|\phi_{\al\beta}*\rho_\beta|_\infty \leq (\e_\al)_+ + \sum_{\beta\in\mathcal{I}}|\phi_{\al\beta}|_\infty M_\beta$, and together with the bound on the spectral gap \eqref{eq:etaSbd}, it follows that the symmetric part $\{S_{\al}\}$ is bounded. 
Finally, together with the vorticity bound \eqref{eq:vorbd} it follows that $\{\pa_j \bu_\al^i\}$ are uniformly bounded which completes the proof. 
\end{proof}

%%%%%%%%%%%%%%%%%%%%%%
\section{Multi-species aggregation dynamics}\label{sec:agg}
%%%%%%%%%%%%%%%%%%%%%
In this section, we prove Theorem \ref{Thm_4}. We begin by letting 
 $\overline{\bx}_\infty(t)$ denote the \emph{center of mass} at time $t$, i.e., 
\begin{equation}\label{Aggregation_eq_center_of_mass}
\overline{\bx}_\infty(t):=\frac{1}{\Min} \sum_{\al\in\mathcal{I}}\overline{\bx}_\al(t), \qquad 
\overline{\bx}_\al(t)= \int_{\rr^d} \rho_\al(t,\bx) \bx d\bx.
\end{equation}
The total  mass $M=\sum_\al \int \rho_\al(t,\bx)\dbx$ is conserved in time. Moreover, by the assumed symmetry of the $\Phi=\{\phi_{\al\bet}\}$ array, the total first moment is also conserved in time,
\begin{align*}
\ddt\sum_{\al\in\mathcal{I}}\int \rho_\al(t,\bx) {\bx} \dbx
= -\iint \sum_{\al,\beta\in\mathcal{I}}  \phi_{\al\beta}(|\bx-\by|)(\bx-\by) \rho_\beta(t,\by)\rho_\al(t,\bx)\dbx \dby=0,
\end{align*}
since the last integrand in anti-symmetric in $(\bx,\by)$. Hence the center of mass is invariant in time $\overline{\bx}_\infty(t)=\overline{\bx}_\infty(0)$.

By assumption, initial densities $\rho_\al(0)$'s are compactly supported. What distinguishes the first-order multi-species aggregation dynamics \eqref{EQ:AggregationEq_multi-groups} is the fact that the diameter of this support does not increase in time, in contrast to the possible expansion  \eqref{eq:Dbd} of $D(t)$ in the second-order flocking dynamics  \eqref{Hydrodynamic_Flocking_eqs}.

\begin{thm}[{\bf Uniformly bounded support}]\label{THM_6_1}\mbox{ }\newline
Consider a strong solution of \eqref{EQ:AggregationEq_multi-groups}, 
$\{\rho_{{\al}}(t,\cdot)\in W^1_+(\rr^d), \al \in \mathcal{I}\}$,
 subject to  compactly supported initial data $\{\rho_{\al 0}\}$. Then the diameter of its support, 
 \[
D(t):=\sup_{\bx,\by\in{\mathcal S}(t)}|\bx-\by|, \qquad
{\mathcal S}(t)=\cup_\al \mathrm{supp}\,\{\rho_\al(t,\cdot)\}
\]
does not increase in time $D(t) \leq D_0$.
\end{thm}
\begin{proof} There are various approaches to trace the diameter $D(t)$ for one-species dynamics, e.g., \cite{BertozziCarrilloLaurent09, CarrilloDiFrancescoFigalliLaurentSlepcev11}. Here we proceed by considering the $p$-weighted diameter ($p$-Wasserstein metric), 
\[
W_p(\rho(t)):= \iint \sum_{\al,\bet\in{\mathcal I}}|\bx-\by|^p\rho_\al(t,\bx)\rho_\bet(t,\by)\dbx\dby.
\]
We abbreviate ${\D}{\sf m}_{\al\bet\gam}(t,\bx,\by,\bz)=\rho_\gam(t,\bz)\rho_\al(t,\bx) \rho_\bet(t,\by)\dbx\dby\dbz$. Differentiation yields
\begin{equation}\label{eq:Wp}
\begin{split}
\hf\ddt &W_p(\rho(t)) = \hf\iint \sum_{\al,\bet\in{\mathcal I}}|\bx-\by|^p\big(\pa_t \rho_\al(t,\bx)\rho_\bet(t,\by)+\rho_\al(t,\bx)\pa_t \rho_\bet(t,\by)\big)\dbx\dby\\
& = -\iiint  \sum_{\al,\bet,\gam \in{\mathcal I}}p |\bx-\by|^{p-2}\langle (\bx-\by), (\bx-\bz)\rangle\phi_{\al\gam}(|\bx-\bz|){\D}{\sf m}_{\al\bet\gam}(t,\bx,\by,\bz).
\end{split}
\end{equation}

The convexity of $|\cdot|^p$ implies
$|\bw-\bv|^p \geq |\bw|^p-p|\bw|^{p-2}\langle \bw,\bv\rangle$ which in turn, setting $\bw=\bx-\by$ and $\bv=\bx-\bz$, shows  that the last integral does not exceed 
\[
\begin{split}
-\iiint  & \sum_{\al,\bet,\gam \in{\mathcal I}}p|\bx-\by|^{p-2}\langle (\bx-\by), (\bx-\bz)\rangle\phi_{\al\gam}(|\bx-\bz|){\D}{\sf m}_{\al\bet\gam}(t,\bx,\by,\bz)\\
& \leq \iiint  \sum_{\al,\bet,\gam \in{\mathcal I}}\left(|\bz-\by|^p-|\bx-\by|^p\right) \phi_{\al\gam}(|\bx-\bz|){\D}{\sf m}_{\al\bet\gam}(t,\bx,\by,\bz)\\
& =  \iiint \sum_{\al,\bet,\gam \in{\mathcal I}}|\bz-\by|^p \phi_{\al\gam}(|\bx-\bz|){\D}{\sf m}_{\al\bet\gam}(t,\bx,\by,\bz)\\
 & \ \ - \iiint \sum_{\al,\bet,\gam \in{\mathcal I}}|\bx-\by|^p \phi_{\al\gam}(|\bx-\bz|){\D}{\sf m}_{\al\bet\gam}(t,\bx,\by,\bz)=: I + II.
\end{split}
\]
Now exchange $\al \leftrightarrow \gam$ and $\bx \leftrightarrow \bz$ in $I$  to conclude that $I+II=0$,  hence $W_p(\rho(t)) \leq W_p(\rho(0))$. In particular, letting $p\uparrow \infty$ yields the desired result $D(t)\leq D_0$.
\end{proof}
The  case $p=2$ deserves special attention: in this case, we can quantify the \emph{strict} decay rate of $W_2(\rho(t))$ in term of the connectivity of the
communication array $\Phi(r)$.  
\begin{thm}[{\bf Decay of weighted diameter}]\label{THM_6_2}\mbox{ }\newline
Consider a strong solution of \eqref{Hydrodynamic_Flocking_eqs}, 
$\{\rho_{{\al}}(t,\cdot)\in W^1_+(\rr^d), \al \in \mathcal{I}\}$,
 subject to  compactly supported initial data $\rho_{\al 0}$ and communication array
 $\Phi_0=\{\phi_{\al\bet}(D_0)\}_{\al,\bet\in{\mathcal I}}$.
Then the weighted diameter  $\delD(t)$ satisfies
\begin{equation}\label{Second_Moment_time_evolution}
\delD(t) \leq e^{\displaystyle -2\zetaM\lambda_2(\DelMP_0)t} \cdot\delD(0) ,\quad \delD(t) =\sum_{\al,\bet\in\mathcal{I}}\iint  |\bx-\by|^2 \rho_\al(t,\bx)\rho_\bet(t,\by) \dbx\dby.
\end{equation}
\end{thm}
\begin{proof}
We begin with computing the time evolution  of $\delD(t)=W_2(\rho(t))$ in \eqref{eq:Wp}: the special case $p=2$ yields, upon exchange $\bx\leftrightarrow \bz$, 
\[
\begin{split}
\ddt\left(\sum_{\al,\bet\in\mathcal{I}}\right.&\left.\iint  |\bx-\by|^2\rho_\al(t,\bx)\rho_\bet(t,\by)\dbx\dby\right) \\
& =-2M\sum_{\al,\beta\in\mathcal{I}} \iint \phi_{\al\beta}(|\bx-\by|)\langle (\bx-\by),2\bx\rangle\rho_\beta (t,\by)\rho_\al(t,\bx)\dbx \dby.
\end{split}
\]
Alternatively, since the center of mass  $\sum_\al \int \rho_\al(t,\bx) \bx \dbx$ is invariant in time,  the  change of the weighted diameter $\displaystyle \ddt \delD(t)$ equals the rate of the total second moment $\displaystyle \sum_{\al\in\mathcal{I}}\int  |\bx|^2 \rho_\al(t,\bx) \dbx$; arguing along the lines of the proof of theorem \ref{THM_4_1} we find
\begin{align*}
\frac{1}{2\Min}\ddt\left(\sum_{\al,\bet\in\mathcal{I}}\right.&\left.\iint  |\bx-\by|^2\rho_\al(t,\bx)\rho_\bet(t,\by)\dbx\dby\right) = \ddt \left(\sum_{\al\in\mathcal{I}}\int  |\bx|^2 \rho_\al(t,\bx) \dbx\right)\\
=&-\sum_{\al,\beta\in\mathcal{I}} \iint \phi_{\al\beta}(|\bx-\by|)\langle (\bx-\by),2\bx\rangle\rho_\beta (t,\by)\rho_\al(t,\bx)\dbx \dby\\
=&-\sum_{\al,\beta\in\mathcal{I}}\ \iint \phi_{\al\beta}(|\bx-\by|)|\bx-\by|^2\rho_\al(t,\bx)\rho_\beta (t,\by)\dbx \dby\\
\leq&-\sum_{\al,\beta\in\mathcal{I}}\phi_{\al\beta}(D_0) \iint|\bx-\by|^2\rho_\al(t,\bx)\rho_\beta (t,\by)\dbx \dby.
\end{align*}
The last step follows from $|\bx-\by|\leq D(t)\leq D_0$ and recalling that $\phi_{\al\bet}$ are decreasing.
Using the vector version of Poincar\'{e} inequality \eqref{eq:alphaneqbeta} with $(\bu_\al(\bx),\bu_\bet(\by))=(\bx,\by)$ we conclude
\begin{align*}
\frac{1}{2\Min}\ddt\sum_{\al,\bet\in\mathcal{I}}\iint |\bx-\by|^2&\rho_\al(t,\bx)\rho_\bet(t,\by)\dbx\dby\\
%%%  correction: = changed to \leq
\leq& -\lambda_2(\DelMP(D_0))\frac{\zetaM}{\Min}\sum_{\al,\bet\in\mathcal{I}}\iint|\bx-\by|^2\rho_\al(t,\bx)\rho_\bet(t,\by) \dbx\dby.
\end{align*}
The bound \eqref{Second_Moment_time_evolution} follows.
\end{proof}

\noindent
{\bf Acknowledgments}. We thank Ruiwen Shu for reading out manuscript and offering the proof for the  improved bound in theorem \ref{THM_4_2}. Research was supported in part by NSF grants 	DMS16-13911, RNMS11-07444 (KI-Net) and ONR grant N00014-1812465.

%\bibliographystyle{abbrv}
%\bibliographystyle{plain}
%\def\cprime{$'$}

%\bibliographystyle{abbrv}
%\bibliography{nonlocal_eqns,JacobBib,SimingBib}

\begin{thebibliography}{30}
\bibitem[BDT17-19]{bellomo2017-19active}
N.~Bellomo, P.~Degond, and E.~Tadmor.
\newblock {\em Active Particles, Volumes 1 \& 2: Advances in Theory, Models,
  and Applications}.
\newblock Birkh{\"a}user, 2017 \& 2019.

\bibitem[BCT09]{BertozziCarrilloLaurent09}
A.~Bertozzi, J.~Carrillo, and T.~Laurent.
\newblock Blow-up in multidimensional aggregation equations with mildly
  singular interaction kernels.
\newblock {\em Nonlinearity}, 22:683--710, 2009.

\bibitem[CCTT16]{CCTT16}
J.~A. Carrillo, Y.-P. Choi, E.~Tadmor, and C.~Tan.
\newblock Critical thresholds in 1d {Euler} equations with nonlocal forces.
\newblock {\em Mathematical Models and Methods in Applied Sciences 26(1)},
  2016.

\bibitem[CDFLS11]{CarrilloDiFrancescoFigalliLaurentSlepcev11}
J.~A. Carrillo, M.~DiFrancesco, A.~Figalli, T.~Laurent, and D.~Slep{c}ev.
\newblock Global-in-time weak measure solutions and finite-time aggregation for
  nonlocal interaction equations.
\newblock {\em Duke Math. J.}, 156, Number 2:229--271, 2011.

\bibitem[CS2007]{CS07}
F.~Cucker and S.~Smale.
\newblock On the mathematics of emergence.
\newblock {\em Japanese Journal of Mathematics}, 2(1):197--227, 2007.

\bibitem[DiFF2013]{FrancescoFagioli2013}
M.~DiFrancesco and S.~Fagioli.
\newblock Measure solutions for non-local interaction pdes with two species.
\newblock {\em Nonlinearity}, 26(10):2777, 2013.

\bibitem[EKLV2017]{KazianouLiaoVauchelet17}
C.~Emako-Kazianou, J.~Liao, and N.~Vauchelet.
\newblock Synchronising and non-synchronising dynamics for a two-species
  aggregation model.
\newblock {\em Discrete \& Continuous Dynamical Systems - B}, 22:2121--2146,
  2017.

\bibitem[EFK2017]{EversFetecauKolokolnikov17}
J.~H.~M. Evers, R.~C. Fetecau, and T.~Kolokolnikov.
\newblock Equilibria for an aggregation model with two species.
\newblock {\em SIAM J. Applied Dynamical Systems}, 16(4):2287--2338, 2017.

\bibitem[FHK2011]{FetecauHuangKolokolnikov11}
R.~C. Fetecau, Y.~Huang, and T.~Kolokolnikov.
\newblock Swarm dynamics and equilibria for a nonlocal aggregation model.
\newblock {\em Nonlinearity}, 24(10):2681, 2011.

\bibitem[Fie1975]{Fiedler75}
M.~Fiedler.
\newblock A property of eigenvectors of nonnegative symmetric matrices and its
  application to graph theory.
\newblock {\em Czech. Math. J.}, 25 (100):619--633, 1975.

\bibitem[GRSB2019]{griffin2019consensus}
C.~Griffin, S.~Rajtmajer, A.~Squicciarini, and A.~Belmonte.
\newblock Consensus and information cascades in game-theoretic imitation
  dynamics with static and dynamic network topologies.
\newblock {\em SIAM Journal on Applied Dynamical Systems}, 18(2):597--628,
  2019.

\bibitem[Gri1988]{Grindrod88}
P.~Grindrod.
\newblock Models of individual aggregation or clustering in single and
  multi-species communities.
\newblock {\em J. Math. Biol. 26:651-660}, 1988.

\bibitem[HKZZ2017]{ha2017emergent}
S.-Y. Ha, D.~Ko, Y.~Zhang, and X.~Zhang.
\newblock Emergent dynamics in the interactions of cucker-smale ensembles.
\newblock {\em Kinetic \& Related Models}, 10(3), 2017.

\bibitem[HT2008]{HaTadmor08}
S.-Y. Ha and E.~Tadmor.
\newblock From particle to kinetic and hydrodynamic descriptions of flocking.
\newblock {\em Kinetic and Related Models}, 1(3):415--435, 2008.

\bibitem[HT2017]{HeTadmor17}
S.~He and E.~Tadmor.
\newblock Global regularity of two-dimensional flocking hydrodynamics.
\newblock {\em Comptes rendus - Mathématique Ser. I 355}, 2017.

\bibitem[HB2010]{HuangBertozzi10}
Y.~Huang and A.~L. Bertozzi.
\newblock Self-similar blowup solutions to an aggregation equation in {$R^n$}.
\newblock {\em SIAM J. Appl. Math., 70(7)}, pages 2582--2603, 2010.

\bibitem[KF2017]{KangFigalli17}
M.~Kang and A.~Figalli.
\newblock A rigorous derivation from the kinetic {C}uker-{S}male model to the
  pressureless {E}uler system with nonlocal alignment.
\newblock {\em arxiv.org/abs/1702.08087}, 2017.

\bibitem[Moh1991]{Mohar91}
B.~Mohar.
\newblock The {L}aplacian spectrum of graphs.
\newblock {\em In ``Graph Theory, Combinatorics, and Applications''},
  2:871--898, 1991.

\bibitem[MT2014]{MotschTadmor14}
S.~Motsch and E.~Tadmor.
\newblock Heterophilious dynamics enhances consensus.
\newblock {\em SIAM Review 56(4)}, 2014.

\bibitem[Pou2002]{Poupaud02}
F.~Poupaud.
\newblock Diagonal defect measures, adhesion dynamics and {Euler} equation.
\newblock {\em Meth. Appl. Anal.}, 9(4):533--562, 2002.

\bibitem[ST2017]{ShvydkoyTadmorI}
R.~Shvydkoy and E.~Tadmor.
\newblock Eulerian dynamics with a commutator forcing.
\newblock {\em Transactions of Mathematics and its Applications 1(1)}, 2017.

\bibitem[TT2014]{TanTadmor}
E.~Tadmor and C.~Tan.
\newblock Critical thresholds in flocking hydrodynamics with non-local
  alignment.
\newblock {\em Philosophical Transactions of the Royal Society A: Mathematical,
  Physical and Engineering Sciences 372:20130401}, 2014.

\end{thebibliography}
\end{document}